\documentclass[11pt]{article}

\usepackage[utf8]{inputenc}
\usepackage{algorithm}
\usepackage{algpseudocode}
\usepackage{amsfonts}
\usepackage{amsmath,amssymb}
\usepackage{amsthm} 
\usepackage{booktabs}
\usepackage{color}
\usepackage{enumerate}
\usepackage{graphicx}
\usepackage{latexsym}
\usepackage{longtable}
\usepackage{mathtools}
\usepackage{multirow}
\usepackage{setspace}
\usepackage{subfigure}
\usepackage{url}
\usepackage[colorlinks=true,breaklinks=true,bookmarks=true,urlcolor=blue,
     citecolor=blue,linkcolor=blue,bookmarksopen=false,draft=false]{hyperref}

\definecolor{myred}{RGB}{255,50,50}              
\definecolor{myblack}{RGB}{0,0,0}                
\newcommand{\red}[1]{\textcolor{myblack}{#1}}

\newcommand{\grad}{\nabla}                       
\newcommand{\inner}[2]{\langle#1,#2\rangle}      
\newcommand{\norm}[1]{\|#1\|}                    
\newcommand{\R}{\mathbb{R}}                      
\newcommand{\T}{\top\hspace{-1pt}}               

\newcommand{\mini}[1]{\underset{{#1}}{\mathrm{minimize}}\ }

\newcommand{\limit}[1]{\underset{{#1}}{\mathrm{lim}}\ }

\usepackage[
  backend=biber,
  giveninits=true,
  style=ieee,
  citestyle=numeric-comp,
  sorting=nyt]{biblatex}

\newtheorem{theorem}{Theorem}[section]
\newtheorem{lemma}[theorem]{Lemma}
\newtheorem{definition}[theorem]{Definition}

\newtheorem{assumption}[theorem]{Assumption}

\newtheorem{remark}[theorem]{Remark}

\numberwithin{equation}{section}

\begin{document}


\title{A second-order sequential optimality condition for\\ nonlinear
    second-order cone programming problems%
    \thanks{This is a pre-print of an article published in Computational Optimization and Applications. The version of record is available online at: \url{https://doi.org/10.1007/s10589-025-00649-0}.}}
\author{
  Ellen H. Fukuda$^{*}$
  \and
  Kosuke Okabe%
  \thanks{Graduate School of Informatics, Kyoto University, Kyoto \mbox{606--8501}, Japan
    (\texttt{ellen@i.kyoto-u.ac.jp}, \texttt{okabe.kosuke.88n@kyoto-u.jp}).}
}

\date{January 23, 2025}

\maketitle


\begin{abstract}
    \noindent 
    In the last two decades, the sequential optimality conditions, which do not require constraint qualifications and allow improvement on the convergence assumptions of algorithms, had been considered in the literature. It includes the work by Andreani et al. (2017), with a sequential optimality condition for nonlinear programming, that uses the second-order information of the problem. More recently, Fukuda et al. (2023) analyzed the conditions that use second-order information, in particular for nonlinear second-order cone programming problems (SOCP). However, such optimality conditions were not defined explicitly. In this paper, we propose an explicit definition of approximate-Karush-Kuhn-Tucker 2 (AKKT2) \red{and complementary-AKKT2 (CAKKT2)} conditions for SOCPs. We prove that the proposed \red{AKKT2/CAKKT2} conditions are satisfied at local optimal points of the SOCP without any constraint qualification. We also present two algorithms that are based on augmented Lagrangian and sequential quadratic programming methods and show their global convergence to points satisfying the proposed conditions.\\

\end{abstract}


\section{Introduction}
We consider the following nonlinear second-order cone programming problem (SOCP):
\begin{equation}
  \label{eq:socp}
  \tag{SOCP}
  \begin{array}{ll}
    \mini{x \in \R^n} & {f(x)} \\
    \mathrm{subject\ to} & g_{i}(x) \in \mathcal{K}_{i}, \quad i = 1, \dots, r,  \\
    & h(x) = 0,
  \end{array}
\end{equation}
where $f\colon \R^n \rightarrow \R,\ h\colon \R^n \rightarrow \R^p,\ g_{i}\colon \R^n \rightarrow \R^{m_i},\ i = 1,\dots, r$ are twice continuously differentiable and $\mathcal{K}_{i}$ denotes the $m_i$ dimensional second-order (or Lorentz) cone, which is defined as 
\[
    \mathcal{K}_{i} := \left\{
    \begin{array}{ll}
    \{(w_0, \overline{w}) \in \R \times \R^{m_i - 1} \mid w_0 \geq \norm{\overline{w}}\},
    & \mbox{if } m_i > 1, \\
    \{w \in \R \mid w \geq 0\}, 
    & \mbox{if } m_i = 1, \\
    \end{array}
\right.
\]
with $\norm{\cdot}$ denoting the Euclidean norm. We write the Cartesian product of Lorentz cones by $\mathcal{K} := \mathcal{K}_1 \times \dots \times \mathcal{K}_r$. Also, $g \colon \R^n \to \R^m$ is given by $g = (g_1,\dots,g_r)$ where $m = \sum_{i=1}^r m_i$. Thus, the conic constraints can be written simply as $g(x) \in \mathcal{K}$.

The nonlinear SOCP is a particular case of nonlinear symmetric cone programming problem and has a wide range of applications, for instance, in finance, control theory, information theory, and others~\cite{LO98}. Because of its practical importance, the optimization community has studied its related theory~\cite{FF16,AHMRS22-1} and proposed various algorithms for nonlinear SOCPs. Some examples of such algorithms are the sequential quadratic programming (SQP) method~\cite{KA07}, the augmented Lagrangian~\cite{LI08}, the exact penalty~\cite{FSF12}, and the interior point method~\cite{YA09}. Most of these methods generate iterates with limit points satisfying 
the first-order optimality conditions.

To solve SOCP, we often aim to find points that satisfy the so-called necessary optimality conditions, which local minimizers of the problem should fulfill. The Karush-Kuhn-Tucker (KKT) conditions are the most well-known first-order necessary optimality conditions, and are satisfied at local optima under some regularity condition which is called constraint qualification (CQ). This means that in the conic context, we usually need to assume, for instance, the Robinson CQ or the nondegeneracy CQ. For optimization problems that lack CQs, it is not possible to guarantee the existence of Lagrange multipliers that satisfy the KKT conditions, and hence it is generally difficult for standard optimization methods to solve them.

Recently, unlike the classical KKT, conditions that use sequences of \red{points} and Lagrange multipliers, called sequential optimality conditions have been proposed. One of such conditions is the Approximate KKT (AKKT) conditions~\cite{AN09}. There are two attractive properties for sequential optimality conditions. First, they are necessary optimality conditions without the requirement of CQs. Second, they have clear relation with the classical optimality conditions, in the sense that they are equivalent when a CQ holds.
In particular, a second-order sequential optimality condition called AKKT2, which \red{incorporates second-order information into the AKKT conditions and is therefore} stronger than AKKT, was introduced in the context of NLP~\cite{AN17}. 
It was also shown that some algorithms generate AKKT2 points. However, it has not been fully extended to SOCP yet.
A second-order analysis for SOCP is more complicated than for NLP, since we should consider the curvature of the second-order cone. Second-order analyses in the conic context were first done by~\cite{KA88}, and later generalized by~\cite{CO90}. Then, second-order necessary optimality conditions for SOCP were proposed in~\cite{BO05}. \red{We refer to~\cite{BS00} for details about optimality conditions for SOCP and general conic optimization.} Moreover, the weak second-order necessary condition (WSONC)~\cite{AN07} was extended to SOCP~\cite{EL20}. \red{Compared to the classical (or basic) second-order necessary condition (BSONC), the WSONC is more practical. In fact, it is known that some algorithms may produce sequences whose limit points fail to satisfy BSONC \cite{AS18,GO99}, whereas WSONC can be fulfilled under weaker assumptions \cite{ABMS10,MP03}.} Although second-order analysis for conic programming was considered in these papers, there is still no formal definition of AKKT2 for those problems.

For this reason, in this paper, we formally define the second-order sequential optimality condition AKKT2 for SOCP extending the AKKT2 for NLP. As the AKKT2 for NLP, we prove three properties that are desirable for a second-order sequential optimality condition. First, a local minimum satisfies the condition under a reasonable assumption. Second, if an AKKT2 point fulfills some CQs, it also satisfies WSONC. Third, there exist algorithms that generate AKKT2 points. In particular, we show that with small modifications, an augmented Lagrangian method and a sequential quadratic programming method generate AKKT2 points under appropriate assumptions. \red{We also see that replacing AKKT with another stronger sequential optimality condition called complementary-AKKT (CAKKT)~\cite{AMS10} into our definition of AKKT2 yields the so-called CAKKT2, which generalizes~\cite{Hae18}. Then, similar results are obtained with the new CAKKT2.} The main contribution here, compared to the  work~\cite{EL20} on second-order optimality conditions for conic problems, can be summarized as follows:
\begin{itemize}
\item Without an explicit definition of AKKT2, it may be difficult to propose and analyze algorithms that generate AKKT2 points. Here, differently from~\cite{EL20}, we are able to provide those algorithms.
\item In~\cite{EL20}, by assuming Robinson CQ and the weak constant rank property (WCR), it is proved that a local minimum satisfies WSONC. Here, since we have an explicit definition of AKKT2 we first (a) prove \red{under an assumption that is not CQ} that a local minimum satisfies AKKT2, and then, (b) assuming Robinson CQ and WCR, it satisfies WSONC. The proof then becomes more complicated here because we are not assuming WCR in (a), which was crucial in the proof given in~\cite{EL20}.
\end{itemize}

This paper is organized as follows. In Section~\ref{sec:preliminaries}, we give some notations and some basic results which will be useful for our work. In Section~\ref{sec:optimality}, we review the existing second-order analyses for SOCP. In Section~\ref{sec:akkt2}, we propose a new second-order sequential optimality condition for SOCP and discuss its properties. In Section~\ref{sec:algorithms}, we propose algorithms that generate points that fulfill the proposed condition. In Section~\ref{sec:conclusion} we give some conclusions.

\section{Preliminaries}
\label{sec:preliminaries}
In this section, we introduce some notations and results that we use throughout the paper. Let us first give the notations. The transpose of a matrix $A$ is denoted by~$A^\T$ \red{ and the identity matrix with dimension $\ell$ is
denoted by $I_\ell$.} We write the set of nonnegative real numbers by~$\R_+$. \red{For any vector $w \in \R^\ell$, we consider the block notation $w := (w_0, \overline{w}) \in \R \times \R^{\ell-1}$, where $w_0$ may also be denoted as $[w]_0$ for clarity.} The inner product of vectors $x, y \in \R^n$ is denoted by $\inner{x}{y} := \sum_{i=0}^{n-1}x_{i}y_{i}$, where~\red{$x := (x_0,\dots,x_{n-1})^\T$ and $y := (y_0,\dots,y_{n-1})^\T$}, respectively. The Euclidean norm of $x \in \R^n$ is written by $\norm{x} := \inner{x}{x}^{\frac{1}{2}}$. For a second-order cone~$\mathcal{K}_i$, we write $\mathrm{int}(\mathcal{K}_i) \red{\,:= \{ (w_0, \overline{w}) \mid w_0 > \norm{\overline{w}}\}}$ and $\mathrm{bd}^{+}(\mathcal{K}_i) \red{\,:= \{ (w_0, \overline{w}) \mid w_0 = \norm{\overline{w}} \ne 0}\}$ as the interior and the boundary excluding the origin of the cone respectively. Using these notations, we define the following sets of indices of \eqref{eq:socp}, which forms a partition of $\{1, \dots, r\}$:
\begin{equation*}
  \begin{array}{rcl}
    I_0(x) & := & \{i \in \{1, \dots, r\} \mid g_i(x) = 0\}, \\
    I_B(x) & := & \{i \in \{1, \dots, r\} \mid g_i(x) \in \mathrm{bd}^{+}(\mathcal{K}_{i})\}, \\
    I_I(x) & := & \{i \in \{1, \dots, r\} \mid g_i(x) \in \mathrm{int}(\mathcal{K}_{i})\}.
  \end{array}
\end{equation*}
Moreover, if a real symmetric matrix $X$ is positive definite (positive semidefinite), it is denoted by $X \succ O$ ($X \succeq O$). The gradient and the Hessian of a function $\zeta\colon \R^n \to \R$ at $x \in \R^n$ are written as $\grad \zeta(x)$ and $\grad^2 \zeta(x)$ respectively. If $\tilde{\zeta}\colon \R^n \times \R^m \to \R$, then its gradient at $(x, y) \in \R^n \times \R^m$ with respect to $x$ is denoted by $\grad_{x} \tilde{\zeta}(x, y)$. Let $\tilde{\varphi} := (\tilde{\varphi}_1, \dots, \tilde{\varphi}_m)^{\T}$ with $\tilde{\varphi}_i: \R^n \to \R$ for all $i = 1, \dots, m$. Then, the Jacobian matrix of $\tilde{\varphi}$ at $x$ is written by $D\tilde{\varphi}(x) = (\grad \tilde{\varphi}_1(x), \dots, \grad \tilde{\varphi}_m(x))^{\T}$. We denote the projection of $x$ onto a second-order cone~$\mathcal{K}_i$ by $\Pi_{\mathcal{K}_i}(x)$, which fulfills
\begin{equation*}
  \norm{x - \Pi_{\mathcal{K}_i}(x)} = \mathrm{min}\big\{\norm{x-z} \mid z \in \mathcal{K}_i\big\}.
\end{equation*}
Note that $\Pi_{\mathcal{K}_i}(x)$ is well-defined since $\mathcal{K}_i$ is a closed convex set.

Next, we introduce some results of nonsmooth analysis~\red{\cite{Cla83,FP03}}. Let $W \colon \R^n \to \R^s$ be a locally Lipschitz function. The B-subdifferential of $W$ at $x \in \R^n$ is defined by
\begin{equation*}
  \partial_{B}W(x) := \{
    V \in \R^{s \times n} \mid
    \exists \{x^k\} \subset \mathcal{D}(W),
    x^k \to x,
    DW(x^k) \to V
  \},
\end{equation*}
where $\mathcal{D}(W)$ is the \red{set where $W$ is differentiable}. Also, we define the Clarke subdifferential of $W$ at $x$, which is the convex hull of $\partial_{B}W(x)$, by $\partial W(x)$. \red{If $W$ is differentiable and $\grad W$ is locally Lipschitz,} the generalized Hessian of $W$ at $x$ is defined by $\partial^2 W(x) := \partial\grad W(x)$. We also give the following result related to generalized Hessians.
\begin{theorem}
  \label{theorem:generalized_hessian}
  \cite[Theorem 3.1]{HI84}
  Let $x^*$ be a local minimizer of a differentiable function $W\colon \R^n \to \R$ and $\grad W$ be locally Lipschitz. Then, $\grad W(x^*) = 0$ holds and \red{for each $d \in \R^n$,} there exists $M \in \partial^2 W(x^*)$ such that $d^{\T}Md \geq 0$.
\end{theorem}
We use this result to get the second-order information of the function associated with the proposed condition. To make a more detailed analysis, we introduce the following theorem, which derives from~\cite[Theorem 5.1]{PA06}.
\begin{theorem}
  \label{theorem:convolution}
  Let $Y\colon \R^n \to \R^s$ and $W\colon \R^s \to \R^s$ be functions such that $Y$ is continuously differentiable at fixed $x \in \R^n$ and $W$ is Lipschitz continuous at a neighborhood of $Y(x)$. Then, we obtain
  \begin{equation*}
    \partial(W \circ Y)(x) \subseteq \partial W(Y(x)) \circ DY(x),
  \end{equation*}
  where $\partial W(Y(x)) \circ DY(x) := \{V \circ DY(x) \mid V \in \partial W(Y(x)) \}$.
\end{theorem}

We now state some important results related to second-order cones.
\begin{lemma}
  \label{lemma:projection}
  \cite[Proposition 3.3]{FU02}
  For any $z = (z_0, \overline{z}) \in \R \times \R^{n-1}$, we have
  \begin{equation*}
    \Pi_{\mathcal{K}}(z) = \mathrm{max}\{0, \eta_1\}u^1 + \mathrm{max}\{0, \eta_2\}u^2,
  \end{equation*}
  where $\eta_1 := z_0 - \norm{\overline{z}}$ and $\eta_2 := z_0 + \norm{\overline{z}}$ \red{are the eigenvalues of $z$}, and
  \begin{equation*}
    u^1 :=
    \left\{
      \begin{array}{ll}
        \frac{1}{2}
        \left(
          \begin{array}{c}
            1 \\
            -\frac{\overline{z}}{\norm{\overline{z}}}
          \end{array}
        \right) &
        \mathrm{if\ } \overline{z} \neq 0, \\
        \frac{1}{2}
        \left(
          \begin{array}{c}
            1 \\
            -\overline{w}
          \end{array}
        \right) &
        \mathrm{otherwise},
      \end{array}
    \right.
    u^2 :=
    \left\{
      \begin{array}{ll}
        \frac{1}{2}
        \left(
          \begin{array}{c}
            1 \\
            \frac{\overline{z}}{\norm{\overline{z}}}
          \end{array}
        \right) &
        \mathrm{if\ } \overline{z} \neq 0, \\
        \frac{1}{2}
        \left(
          \begin{array}{c}
            1 \\
            \overline{w}
          \end{array}
        \right) &
        \mathrm{otherwise},
      \end{array}
    \right.
  \end{equation*}
  \red{are the eigenvectors of $z$}, where $\overline{w} \in \R^{n-1}$ is any vector which satisfies $\norm{\overline{w}} = 1$. 
\end{lemma}

The lemma above shows that the projection onto a second-order cone can be written simply by taking the maximum between zero and the eigenvalues $\eta_1$ and $\eta_2$ associated with the point~$z$. Moreover, the eigenvectors $u^1$ and $u^2$ of $z$ remain untouched. \red{Now, let us define the matrix} $M_{i} \colon \R \times \R^{m_i-1} \to \R^{m_i \times m_i}$ as
  \begin{equation}
    \label{eq:m_i}
    M_{i}(\xi, w) := \frac{1}{2}\left(\begin{array}{lc} 1 & w^{\T} \\ w & (1+\xi)I_{m_i-1} - \xi ww^{\T} \end{array}\right).
  \end{equation}
The next result characterizes the elements of the B-subdifferential $\partial_{B}\Pi_{\mathcal{K}_i}(z)$ depending on~$z$. 

\begin{lemma}
  \label{lemma:bsub}
  \cite[Proposition 4.8]{HA05}, \cite[Lemma 14]{SU03}
  \red{Let $M_{i}$ be defined in (\ref{eq:m_i}).} The B-subdifferential $\partial_{B}\Pi_{\mathcal{K}_i}(z)$ at $z \in \R^{m_i}$ is given as follows:
  \begin{enumerate}[(a)]
    \item If $z \in int(-\mathcal{K}_{i})$, then $\partial_{B}\Pi_{\mathcal{K}_i}(z) = \{0\}$;
    \item If $z \in int(\mathcal{K}_{i})$, then $\partial_{B}\Pi_{\mathcal{K}_i}(z) = \{I_{m_i}\}$;
    \item If $z \notin \mathcal{K}_{i} \cup -\mathcal{K}_{i}$, then $\partial_{B}\Pi_{\mathcal{K}_i}(z) = \left\{M_{i}\left(\frac{z_0}{\norm{\overline{z}}}, \frac{\overline{z}}{\norm{\overline{z}}}\right)\right\}$;
    \item If $z \in \mathrm{bd}^{+}(\mathcal{K}_{i})$, then $\partial_{B}\Pi_{\mathcal{K}_i}(z) = \left\{I_{m_i}, M_{i}\left(1, \frac{\overline{z}}{\norm{\overline{z}}}\right)\right\}$;
    \item If $z \in \mathrm{bd}^{+}(-\mathcal{K}_{i})$, then $\partial_{B}\Pi_{\mathcal{K}_i}(z) = \left\{0, M_{i}\left(-1, \frac{\overline{z}}{\norm{\overline{z}}}\right)\right\}$;
    \item If $z = 0$, then $\partial_{B}\Pi_{\mathcal{K}_i}(z) = \{0, I_{m_i}\} \cup \{M_{i}(\xi, w) \mid |\xi| \leq 1, \norm{w} = 1\}$;
  \end{enumerate}
\end{lemma}
\begin{lemma}
  \label{lemma:m_eigenvalue}
  \cite[Lemma 2.8]{KA09}
  Let $w \in \R^{m_{i}-1}$ \red{and $\xi \in \R$ such that} $\norm{w} = 1$ and \red{$|\xi| \le 1$.}
  Then, $M_{i}(\xi, w)$ defined in (\ref{eq:m_i}) has the two single eigenvalues $0$ and $1$ as well as $\frac{1}{2}(1+\xi)$ with multiplicity $m_i-2$.
\end{lemma}
We finally, observe from Lemma~\ref{lemma:m_eigenvalue} that the eigenvalues of the matrix $M_{i}(\xi, w)$ are always in the interval $[0, 1]$ when $\norm{w} = 1$ and \red{$|\xi| \le 1$}.

\section{Existing optimality conditions}
\label{sec:optimality}

In this section, we review the existing optimality conditions of SOCP. We start with the following first-order optimality conditions.
\begin{definition}
  We say that $x \in \R^n$ satisfies the KKT conditions if $x$ is a feasible point of \eqref{eq:socp} and there exist $\mu \in \R^p$ and $\omega \red{\,:= (\omega_1,\dots,\omega_r)} \in \mathcal{K}$ which satisfy
  \begin{equation*}
    \grad_x L(x, \mu, \omega) = 0,\qquad \red{\inner{g_i(x)}{\omega_i} = 0 \quad \mbox{for all } i = 1,\dots,r,}
  \end{equation*}
  where $L(x, \mu, \omega) := f(x) + \inner{h(x)}{\mu} - \inner{g(x)}{\omega}$.
\end{definition}
We recall that $L\colon \R^n \times \R^p \times \mathcal{K} \to \R$ is the Lagrange function and $\mu$ and $\omega$ are the Lagrange multipliers associated with the equality and the conic constraints, respectively. If $(x, \mu, \omega)$ satisfies the KKT conditions, we say that $x$ is a KKT point. This is one of the most usual necessary conditions for optimality, however, the KKT conditions do not necessarily hold at a local minimum where constraint qualifications (CQ) are violated. One example of CQ is given in the definition below.
\begin{definition}
  Let $x \in \R^n$ be a feasible point of \eqref{eq:socp}. We say that it satisfies Robinson CQ if $Dh(x)$ has full row rank and there exists some $d \in \R^n$ such that
  \begin{equation*}
    g_i(x) + Dg_i(x)d \in \mathrm{int}(\mathcal{K}_i),\quad Dh(x)d = 0,
  \end{equation*}
  for all $i \in \{1, \dots, r\}$.
\end{definition}
Note that Robinson CQ is a generalization of the Mangasarian-Fromovitz constraint qualification (MFCQ) for NLP. This CQ requires the feasible set of the problem to be similar to its first-order approximation around $x$. This restriction does not necessarily hold, which makes the KKT conditions not always suitable.

To overcome the above problem, sequential optimality conditions have been proposed in the literature. These are alternative optimality conditions that do not require constraint qualifications. The most popular of these conditions, the Approximate KKT (AKKT)\red{,} is introduced in~\cite{AN09} for nonlinear programming and extended to SOCP in~\cite{AN19}. We introduce the AKKT conditions for \eqref{eq:socp} below.
\begin{definition}
  We say that $x^* \in \R^n$ satisfies the AKKT conditions if $x^*$ is a feasible point of \eqref{eq:socp} and there exist sequences $\{x^k\} \subset \R^n, \{\mu^k\} \subset \R^p$ and $\{\omega^k\} \subset \mathcal{K}$ which satisfy
  \begin{align}
    \underset{k \to \infty}{\mathrm{lim}} x^k = x^*,\ 
    \underset{k \to \infty}{\mathrm{lim}} \grad_x L(x^k, \mu^k, \omega^k) = 0, 
    \label{eq:akkt1} \\
    i \in I_{I}(x^*) \Rightarrow \omega^k_i \to 0\ \mathrm{and} \nonumber \\
    i \in I_{B}(x^*) \Rightarrow \omega^k_i \to 0\ \mathrm{or}\ \omega^k_i \in \mathrm{bd}^{+}(\mathcal{K}_i)\ \mathrm{with}\ \frac{\overline{\omega^k_i}}{\norm{\overline{\omega^k_i}}} \to -\frac{\overline{g_i(x^*)}}{\norm{\overline{g_i(x^*)}}}, \nonumber
  \end{align}
  where $g_i(x^*) := ([g_i(x^*)]_0, \overline{g_i(x^*)})$ and \red{$\omega_i^k := ([\omega_i^k]_0, \overline{\omega_i^k}) \in \mathcal{K}_i$}.
\end{definition}
It was shown that a local minimizer of \eqref{eq:socp} satisfies the AKKT conditions and some practical algorithms can generate an AKKT point. Moreover, AKKT implies KKT if Robinson CQ is satisfied (see~\cite{AN19}). 

\red{Another well-known sequential optimality condition is the so-called complementary-AKKT (CAKKT). Its definition is similar to that of AKKT, 
except that the conditions involving the limits in~\eqref{eq:akkt1} remain unchanged, while the rest is replaced by
\begin{align*}
  g_i(x^k) \circ \omega^k_i \to 0 & \quad i=1,\dots,r, \\
  h_i(x^k) \mu_i^k \to 0 & \quad i = 1,\dots,p,
\end{align*}
where $g_i(x) \circ \omega_i$ is the Jordan product of $g_i(x)$ and $\omega_i$, i.e., $g_i(x) \circ \omega_i := (\inner{g_i(x)}{\omega_i}, [g_i(x)]_0 \overline{\omega_i} + [\omega_i]_0 \overline{g_i(x)})$.} 
\red{The above definition is defined for nonlinear programming in~\cite{AMS10} and extended for SOCP problems in~\cite{AN19}. Like the AKKT, local minimizes of~\eqref{eq:socp} also satisfies CAKKT, and practical SOCP algorithms can be shown to generate CAKKT points~\cite{AN19,OYF23}. Also, in~\cite[Section~3]{AN19}, it was shown that CAKKT implies AKKT while the converse is not true.}

We now state the following result that shows the relation of AKKT \red{(and CAKKT)} and a penalty function~\cite[Theorems 3.1 \red{and 3.2}]{AN19}.
\begin{theorem}
  \label{theorem:akkt}
  Let $x^* \in \R^n$ be a local \red{minimizer} of \eqref{eq:socp}. Then, for any given sequence $\{\rho_k\} \to +\infty$, there exists a sequence $\{x^k\} \to x^*$ such that each $x^k$ is a local minimizer of the penalty function:
  \begin{equation}
    \label{eq:penalty_function}
    F_k(x)
    := 
      f(x)
    + \frac{1}{4}\norm{x-x^*}^4
    + \frac{\rho_k}{2}
      \left(
        \sum_{i=1}^{r}\norm{\Pi_{\mathcal{K}_i}(-g_i(x))}^2 + \norm{h(x)}^2
      \right).
  \end{equation}
  Also, letting $\mu^k := \rho_k h(x^k)$ and $\omega^k_i := \rho_k\Pi_{\mathcal{K}_i}(-g_i(x^k))$, then $(x^k, \mu^k, \omega^k)$ satisfies \red{both the AKKT and the CAKKT conditions.}
\end{theorem}
%

As a way to obtain more information, we can consider the second-order conditions of the problem. To discuss second-order conditions of SOCP, we introduce some concepts associated with second-order analyses. First, we define the weak constant rank (WCR) property that will be necessary to introduce second-order conditions.
\begin{definition}
  \label{definition:wcr}
  We say that the weak constant rank (WCR) property holds at a feasible point $x^* \in \R^n$ of \eqref{eq:socp} if there exists a neighborhood $\mathcal{N}$ of $x^*$ such that
  \begin{equation*}
    \{\grad h_i(x)\}_{\{i=1,\dots,p\}}
    \: \bigcup \: \{\grad g_{ij}(x)\}_{i \in I_0(x^*), \: j \in \{1,\dots,m_i\}}
    \: \bigcup \: \{Dg_i(x)^{\T}\Gamma_i\tilde{g}_{i}(x)\}_{i \in I_{B}(x^*)}
  \end{equation*}
  has the same rank for all $x \in \mathcal{N}$, where $\tilde{g}_{i}(x) := (\norm{\overline{g_{i}(x)}}, \overline{g_{i}(x)})$ and
  \begin{equation}
    \label{eq:gamma}
    \Gamma_i := \left(\begin{array}{cc} 1 & 0^{\T} \\ 0 & -I_{m_i-1} \end{array}\right)
  \end{equation}
  for all $i \in \{1,\dots,r\}$.
\end{definition}
\red{When \eqref{eq:socp} is indeed an NLP, i.e., $m_i=1$ for all $i$, the above WCR property becomes equivalent to the WCR of NLP introduced in~\cite{AN07}. Note also that WCR is not a CQ on its own~\cite[Exampple 5.1]{AN07}. Moreover, a well-known CQ for conic optimization is nondegeneracy, which can be seen as a generalization of the linear independence CQ (LICQ) for NLP problems, as it also implies the existence of a unique Lagrange multiplier.
From~\cite{EL20}, we can note that} nondegeneracy implies Robinson CQ and WCR, a property that also holds in NLP. Furthermore, there exist examples of points satisfying Robinson CQ and WCR but not nondegeneracy. Therefore, the joint condition ``Robinson CQ and WCR'' is strictly weaker than nondegeneracy~\cite[Example 5.2]{AN07}.

The sigma-term introduced in~\cite{CO90}, which represents a possible curvature of $\mathcal{K}$ at $g(x)$, is important to consider second-order conditions of SOCP. We define the sigma-term associated to \red{$x^*$} as follows:
\begin{equation}
  \label{eq:sigma-term}
  \sigma(x, \omega) := \sum_{i \in I_{B}(x^*)} \sigma_{i}(x, \omega),
\end{equation}
where
\begin{equation*}
  \sigma_{i}(x, \omega)
  := 
    -\frac{[\omega_{i}]_{0}}{[g_{i}(x)]_{0}}
    Dg_{i}(x)^{\T}\Gamma_{i}Dg_{i}(x) \quad 
    \mathrm{for\ all\ } i \in I_{B}(x^*).
\end{equation*}
Note that \red{the term $\sigma(x, \omega)$ does not appear in NLP.} In fact, the difficulty of generalization of second-order analyses from NLP to SOCP arises from this sigma-term.

As a classical second-order condition, we state the basic second-order necessary condition (BSONC), which is satisfied if Robinson CQ holds at a local minimum $x^*$ of SOCP as follows.
\begin{definition}
  Let $x^* \in \R^n$ be a feasible point of \eqref{eq:socp}. We say that $x^*$ satisfies BSONC if for every $d \in C(x^*)$ there are Lagrange multipliers $\mu^*_d \in \R^p$ and $\omega^*_d \in \mathcal{K}$ such that $(x^*, \mu^*_d, \omega^*_d)$ is a KKT \red{triple} and
  \begin{equation*}
    d^{\T}(\grad^2_x L(x^*, \mu^*_d, \omega^*_d) + \sigma(x^*, \omega^*_d))d \geq 0,
  \end{equation*}
  where $C(x^*)$ is the critical cone of \eqref{eq:socp} at $x^*$.
\end{definition}
The critical cone of \eqref{eq:socp} at $x^*$ is denoted by
\[
  C(x^*) := \left\{ d \in \R^n \, \middle| \,
  \begin{array}{l}
     \inner{\grad f(x^*)}{d} = 0; Dh(x^*)^{\T}d = 0; \\
     Dg_i(x^*)^{\T}d \in T_{\mathcal{K}_i}(g_i(x^*)) \: \forall i \in \{1, \dots, r\}
  \end{array}
  \right\},
\]
where
\begin{equation*}
  T_{\mathcal{K}_i}(g_i(x^*)) := \left\{d \in \R^{m_i} \middle|
  \begin{array}{l}
    \exists\{d^k\}_{k \in \mathbb{N}} \to d, \exists\{\alpha_k\}_{k \in \mathbb{N}} \subset \R_{+}, \alpha_k \to 0, \\
    \forall k \in \mathbb{N}, g_i(x^*) + \alpha_k d^k \in \mathcal{K}_i
  \end{array}
  \right\},
\end{equation*}
is the tangent cone of $\mathcal{K}_i$ at $g_i(x^*)$. One of the disadvantages of BSONC is that we must find $\mu^*_d$ and $\omega^*_d$ which rely on $d$ for all $d \in C(x^*)$, which is not always practical to calculate. Moreover, there is an example that practical algorithms fail to find a BSONC point even in a simple case~\cite[Section 2]{GO99}.

A more practical optimality condition than BSONC is the weak second-order necessary condition (WSONC), which is defined as follows.
\begin{definition}
  \label{definition:wsonc}
  Let $(x^*, \mu^*, \omega^*) \in \R^n \times \R^p \times \mathcal{K}$ be a KKT \red{triple}. We say that $x^*$ satisfies WSONC if
  \begin{equation*}
    d^{\T}(\grad^2_{x} L(x^*, \mu^*, \omega^*) + \sigma(x^*, \omega^{*}))d \geq 0
  \end{equation*}
  for all $d \in S(x^*)$ where 
  \begin{equation*}
    S(x^*)
    :=
    \left\{d \in \R^n \middle|
      \begin{array}{l}
        Dh(x^*)d = 0; Dg_{i}(x^*)d = 0, i \in I_{0}(x^*); \\
        g_{i}(x^*)^{\T}\Gamma_{i}Dg_{i}(x^*)d = 0, i \in I_{B}(x^*)
      \end{array}
    \right\}.
  \end{equation*}
\end{definition}
Note that $S(x^*)$ is equivalent to the largest subspace contained in $C(x^*)$, namely, $S(x^*) = C(x^*) \cap -C(x^*)$. Since Lagrange multipliers in WSONC are independent of $d \in S(x^*)$, it is easier to verify than BSONC. 
\red{While some authors question the existence of an algorithm that considers the entire critical cone $C(x^*)$, WSONC has been shown to be satisfied under weaker assumptions, for example, for barrier-type and augmented Lagrangian methods~\cite{ABMS10, MP03}.} Although WSONC is a more practical condition, there is still a restriction related to CQs. In fact, Robinson CQ alone is not enough to guarantee fulfillment of WSONC at a local minimizer of SOCP~\cite{BA04}. Therefore, WSONC requires a stronger CQ such as nondegeneracy or LICQ as an optimality condition. Recently, it has been shown that the strictly weaker condition ``Robinson CQ and WCR'' is enough for the purpose~\cite[Theorem 6]{EL20}, which is a natural extension of the result in NLP~\cite[Theorem 3.1]{AN07}.

\section{Second-order sequential optimality conditions}
\label{sec:akkt2}

In this section, we extend concepts of sequential second-order optimality conditions for NLP to SOCP. As it was mentioned before, the condition should have three properties: (i) it is an optimality condition that does not require CQ, (ii) it implies WSONC for a weak constraint qualification, (iii) its validity can be verified in sequences generated by practical algorithms. We define our sequential second-order optimality condition for SOCP as follows.
\begin{definition}
  \label{def:akkt2}
  We say that the \red{feasible} point $x^*$ satisfies the AKKT2 conditions \red{if there exist sequences} $\R^n \supset \{x^k\} \to x^*$, $\{\mu^k\} \subset \R^p$ and $\{\omega^k\} \subset \mathcal{K}$ \red{satisfying} the AKKT conditions and, \red{in addition,}
  there exist  $\{\eta^k\} \subset \R^p,\ \{\theta_{i}^k\}_{i \in I_{0}(x^*)} \subset \R, \{\gamma_{i}^{k}\}_{i \in I_{B}(x^*)} \subset \R$ and $\R_{+} \supset \{\delta^k\} \to 0$ such that
  \begin{align*}
      d^{\T}\left(
      \grad^{2}_{x}L(x^k, \mu^k, \omega^k)
      + \sigma(x^k, \omega^k)
      + \sum_{i=1}^{p}\eta^{k}_{i}\grad h_i(x^k)\grad h_i(x^k)^{\T}\right. \\
      + \sum_{i \in I_{B}(x^*)}
      \gamma_i^k
      (Dg_{i}(x^k)^{\T}\Gamma_{i}\tilde{g}_{i}(x^k))
      (Dg_{i}(x^k)^{\T}\Gamma_{i}\tilde{g}_{i}(x^k))^{\T} \\
      + \sum_{i \in I_{0}(x^*)} \theta_{i}^{k}Dg_{i}(x^k)^{\T}Dg_{i}(x^k) 
      + \delta_{k}I_n \Bigg)d \geq 0 
      \quad 
  \end{align*}
  for all $d \in \R^n$ when $k$ is sufficiently large, and where $\tilde{g}_{i}(x^k) := \big(\norm{\overline{g_{i}(x^k)}}, \overline{g_{i}(x^k)}\big)^{\T}$.
  %
\end{definition}
We call $\{x^k\}$ \red{satisfying the conditions of Definition~\ref{def:akkt2} an AKKT2 sequence}. We can easily show that AKKT2 for \eqref{eq:socp} is a generalization of AKKT2 for nonlinear programming~\cite[Definition 3.1]{AN17}, considering $I_0(x^*)$ as the set of active indices at $x^*$ and $I_{B}(x^*) = \emptyset$. We point out that $\delta_k I_n$ in the definition can be replaced with $\Delta_k \to \red{0}$ where $\Delta_k \in \R^{n \times n}$ to have a more general definition. \red{We also note that a definition for CAKKT2 can be similarly provided by simply replacing the first-order conditions AKKT 
with CAKKT. This had been done for nonlinear programming in~\cite{Hae18},
except for the strong complementarity measure.}

In the following, we discuss properties of AKKT2 for \eqref{eq:socp}. Let us first show the following lemmas which will be used in our main results.
\begin{lemma}
  \label{lemma:inner_semicontinuity}
  Let $x^* \in \R^n$ be a feasible point of~\eqref{eq:socp} and let $\{x^k\} \to x^*$. Define
  \begin{equation*}
    \red{S(x, x^*)}
    :=
    \left\{
      d \in \R^n \, \middle| \,
      \begin{array}{l}
        Dh(x)d = 0; Dg_{i}(x)d = 0, i \in I_{0}(x^*); \\
        \tilde{g}_{i}(x)^{\T}\Gamma_{i}Dg_{i}(x)d = 0, i \in I_{B}(x^*)
      \end{array}
    \right\},
  \end{equation*}
   and let the mapping $x \mapsto S(x, x^*)$ be inner semicontinuous at $x^*$. Then, \red{for each $d \in S(x^*)$}, there is a sequence $\{d^k\} \subset S(x^k, x^*)$ which fulfills $\{d^k\} \to d$. Moreover, the mapping $x \mapsto S(x, x^*)$ is inner semicontinuous at $x^*$ if and only if the WCR property holds at $x^*$.
\end{lemma}
\noindent
\begin{proof}
The second statement is shown in~\cite[Lemma 1]{EL20}.
By inner semicontinuity of $x \mapsto S(x^k, x^*)$, we have
\begin{equation*}
  S(x^*) \subseteq \left\{ d \in \R^n \, \middle| \, \limit{y \to x^*}\text{sup\ dist} [d, S(y, x^*)] = 0 \right\},
\end{equation*}
where $\text{dist} [d, S(y, x^*)]$ denotes the distance of $d$ to the set $S(y, x^*)$. Recalling that $d \in S(x^*)$ holds, we obtain
\begin{equation*}
  0 = \limit{y \to x^*}\text{sup\ dist} [d, S(y, x^*)] = \limit{n \to \infty}\text{sup}\left\{\text{dist}[d, S(y, x^*)] \, \middle| \, \norm{y-x^*} \leq \frac{1}{n}\right\}.
\end{equation*}

Now, let us define $t_n := \text{sup} \{\text{dist}[d, S(y, x^*)] \mid \norm{y-x^*} \leq 1/n\}$. It is obvious that $t_n \to 0$ when $n \to \infty$.
Since $x^k \to x^*$, there exists a positive integer $k_n \geq n$ such that $\norm{x^{k_n} - x^*} \leq 1/n$. Therefore, $\text{dist}[d, S(x^{k_n}, x^*)] \leq t_n$ holds.
From the definition of distance, we also have $\text{dist}[d, S(x^{k_n}, x^*)] := \text{inf}\{\norm{d-s} \mid s \in S(x^{k_n}, x^*)\}$. Taking this into account, there exists $d_n \in S(x^{k_n}, x^*)$ such that
\begin{align*}
  \norm{d_n - d} & <
  \text{inf} \{\norm{d-s} \mid s \in S(x^{k_n}, x^*)\} + \frac{1}{n} \\
  & =
  \text{dist}[d, S(x^{k_n}, x^*)] + \frac{1}{n}.
\end{align*}
Therefore, we obtain $\norm{d_n - d} < t_n + 1/n \to 0$ when $n \to \infty$. In conclusion, we can take a subsequence $\{x^{k_n}\} \subset \{x^k\}$ such that there exists $S(x^{k_n}, x^*) \supset \{d_n\} \to d$.
\end{proof}

\begin{lemma}
    \label{lemma:psd_matrix}
    Let $\beta, \xi \in \R$ and $b \in \R^{s-1}$. Then, the matrix
    \[
    P = 
    \left[
    \begin{array}{cc}
         \beta & b^\T \\
         b & \displaystyle{\frac{1}{\beta} bb^\T} + \xi I_{s-1}
    \end{array}
    \right]
    \]
    is positive semidefinite when $\xi \ge 0$ and $\beta > 0$. \red{Additionally, if $\xi > 0$, then $P$ is positive definite.}
\end{lemma}

\begin{proof}
    \red{Assume that $\xi \ge 0$ and $\beta > 0$, and} let us prove that $u^\T P u \ge 0$ for all $u = (u_0, \bar{u}) \in \R^s$. Simple calculations show that
    \begin{align*}
        u^\T P u & = \beta u_0^2 + 2 u_0 b^\T \bar{u} 
        + \frac{1}{\beta} \bar{u}^\T b b^\T \bar{u} + \xi \norm{\bar{u}}^2 \\
        & = \left( \sqrt{\beta} u_0 + \frac{1}{\sqrt{\beta}} \big( b^\T \bar{u} \big) \right)^2
        + \xi \norm{\bar{u}}^2 \ge 0,
    \end{align*}
     which shows that \red{$P \succeq 0$. Now, assume that $\xi > 0$ and take $u = (u_0, \bar{u}) \ne 0$. If $\bar{u} = 0$, then $u_0 \ne 0$ and thus the above inequality holds strictly. When $\bar{u} \ne 0$, we also conclude that $u^\T P u > 0$, which means that $P \succ 0$.}
\end{proof}

\red{Before presenting the main result, we discuss some assumptions. To do so, we introduce the following generalized Lojasiewicz inequality, which is a weak assumption on the smoothness of the constraint functions~\cite{AMS10}.}
\begin{definition}
  \label{def:los}
  \red{A function $Q$ satisfies the generalized Lojasiewicz inequality at $x^*$ if there exists $\delta > 0$ and $\vartheta \colon B(x^*,\delta) \to \R$ satisfying $\lim_{x \to x^*} \vartheta(x) = 0$ and 
  \[
    |Q(x) - Q(x^*)| \le \vartheta(x) \norm{\nabla Q(x)},
  \]
  where $B(x^*, \delta)$ is the Euclidean ball with radius $\delta$ 
  around $x^*$.}
\end{definition}

\begin{lemma}
  \label{lem:los}
  \red{Let $x^* \in \R^n$ be a local minimizer of \eqref{eq:socp}, and $\{\rho_k\} \to +\infty$. Take sequences $\{x^k\} \to x^*$, $\{ \mu^k \}$ and $\{\omega_i^k\}$ as in Theorem \ref{theorem:akkt}. Assume that the feasibility measure of~\eqref{eq:socp} given by
  \begin{equation}
    \label{eq:Q}
    Q(x) := \frac{1}{2}\norm{h(x)}^2 + \frac{1}{2}\norm{\Pi_{\mathcal{K}}(-g(x))}^2
  \end{equation}  
  satisfies the generalized Lojasiewicz inequality at $x^*$.
  Then, for all $i$, we have
  \[
  \lim_{k \to \infty} \frac{(\mu_i^k)^2}{\rho_k} = 0 \quad \mbox{and} \quad 
  \lim_{k \to \infty} \frac{\norm{\omega_i^k}^2}{\rho_k} = 0.
  \]}
\end{lemma}

\begin{proof}
  \red{For sufficiently large $k$, we have $\lim_{k \to \infty} \vartheta(x^k) = 0$ and
  \[
    |Q(x^k) - Q(x^*)| \le \vartheta(x^k) \norm{\nabla Q(x^k)}.
  \]
  Since $x^*$ is feasible, $Q(x^*) = 0$. Thus, from the definition of $Q$, $\mu^k$ and $\omega^k$, we obtain
  \begin{align}
    \frac{1}{2\rho_k} (\norm{\mu^k}^2 + \norm{\omega^k}^2) = & \,
    \frac{1}{2\rho_k} \left( \norm{\rho_k h(x^k)}^2 + \sum_{i=1}^r \norm{\rho_k \Pi_{\mathcal{K}_i}(-g_i(x^k))}^2 \right) \nonumber \\
    \le & \,\vartheta(x^k) \norm{\rho_k \nabla Q(x^k)}. \label{eq:bound_mult}
  \end{align}
  Computing $\nabla Q(x^k)$, we also have
  \begin{align*}
     \norm{\rho_k \nabla Q(x^k)} &  
    = \left\| \rho_k \left( Jh(x^k)^\T h(x^k) - \sum_{i=1}^r Jg_i(x^k)^\T \Pi_{\mathcal{K}_i}(-g_i(x^k)) \right) \right\| \\
    & = \norm{Jh(x^k) \mu^k - Jg(x^k)^\T \omega^k} \\
    & = \norm{\nabla_x L(x^k,\mu^k,\omega^k) - \nabla f(x^k)} \\
    & \le \norm{\nabla_x L(x^k,\mu^k,\omega^k)} + \norm{\nabla f(x^k)}.
  \end{align*}
  From Theorem~\ref{theorem:akkt}, $x^*$ satisfies the AKKT conditions, and from \cite[Lemma 3.1]{AN19}, in particular, there exists $\varepsilon_k > 0$ with $\varepsilon_k \to 0$ such that $\norm{\nabla_x L(x^k,\mu^k,\omega^k)} \le \varepsilon_k$. Thus,~\eqref{eq:bound_mult} can be written as
  \[
    \frac{1}{2\rho_k} (\norm{\mu^k}^2 + \norm{\omega^k}^2) \le \vartheta(x^k)
    (\varepsilon_k + \norm{\nabla f(x^k)}).
  \]
  Since $\norm{\nabla f(x^k)}$ is bounded, we conclude that the left-hand side of the above inequality converges to~$0$, which completes the proof.}
\end{proof}

To prove our main theorem, we consider the following assumption.
\begin{assumption}
  \label{assum:bounded}
  Let $x^* \in \R^n$ be a local minimizer of \eqref{eq:socp}.
  Then, one of the following assumptions holds:
  \begin{enumerate}[(a)]
  \item Let $\{\rho_k\} \to +\infty$, $\{x^k\} \to x^*$ and $\{\omega_i^k\}$ as in Theorem \ref{theorem:akkt}. Then, $\{ [\omega^k_i]_0 \}_{k \in \tilde N}$ 
  is bounded for all $i \in I_B(x^*)$, where
  $\tilde{N} := \{ k \mid g_i(x^k) \in \R^{m_i} \setminus
  (\mathcal{K}_i \cup -\mathcal{K}_i)$\}; or
  \item \red{The feasibility measure~$Q$ defined in~\eqref{eq:Q} satisfies the generalized Lojasiewicz inequality at~$x^*$.}
  \end{enumerate}
\end{assumption}
\noindent The above assumption will be used in Lemma~\ref{lemma:akkt2_inequality} below and discussed further in Remark~\ref{rem:lemmaobst}.
\red{Note that (a)} states that when the sequence $\{g_i(x^k)\}$ is converging 
to a point of the boundary $g_i(x^*)$ from the outside of the cone, then the corresponding first-term of the multiplier $[\omega^k_i]_0 = \rho_k [\Pi_{\mathcal{K}_i}(-g_i(x^k))]_0$ should be bounded. \red{On the other hand, the assumption (b) only involves the problem's constraints, instead of the sequence $\{ x^k \}$, and is known to be a weak condition, since it is satisfied when the constraint functions are analytic~\cite{AMS10}.}

The following lemma is the crucial part of the theorem that shows that AKKT2 is in fact a second-order necessary optimality condition. We state it as a new lemma just for clarity and because its proof is long enough.

\begin{lemma}
  \label{lemma:akkt2_inequality}
  Let $x^* \in \R^n$ be a local \red{minimizer} of~\eqref{eq:socp}.
  Let $\{x^k\} \subset \R^n$ and $\{\rho_k\} \subset \R_{+}$ be sequences such that $\{x^k\} \to x^*$ and $\{\rho_k\} \to +\infty$. Let $V_{i}^{k}$ be an arbitrary element in $\partial\Pi_{\mathcal{K}_{i}}(-g_{i}(x^k))$ and $\omega^k_i := \rho_k\Pi_{\mathcal{K}_i}(-g_i(x^k))$.
  Suppose also that Assumption~\ref{assum:bounded} is satisfied.
  Then, the inequality below holds:
  \begin{align*}
    & d^{\T}\left(\sum_{i=1}^{r} \rho_{k}Dg_{i}(x^k)^{\T}V_{i}^{k}Dg_{i}(x^k)\right)d \\
    \leq \ &
    d^{\T}\left(\sigma(x^k, \omega^k)
    +
    \sum_{i \in I_{0}(x^*)} \theta_{i}^{k} Dg_{i}(x^k)^{\T}Dg_{i}(x^k)\right. \\
    & \qquad + \left.
    \sum_{i \in I_{B}(x^*)} \gamma_{i}^{k} Dg_{i}(x^k)^{\T}\Gamma_{i}\tilde{g}_{i}(x^k)(Dg_{i}(x^k)^{\T}\Gamma_{i}\tilde{g}_{i}(x^k))^{\T}
    + \delta_{k} I_n \right)d
    \end{align*}
  for all $d \in \R^n$ \red{and for all $k$}, with appropriate parameters $\{\theta_{i}^{k}\} \subset \R$, $\{\gamma_{i}^{k}\} \subset \R$ and $\{\delta_{k}\} \subset \R_{+}$ with $\delta_k \to 0$ for sufficiently large $k$.
\end{lemma}
\begin{proof}
  Take an arbitrary $d \in \R^n$ and denote $u_i^k := Dg_i(x^k)d$ for simplicity. Then, it is sufficient to show that
\begin{align}
  \label{eq:socp_akkt2_target_2}
  \sum_{i=1}^{r} \rho_{k}(u_{i}^{k})^{\T}V_{i}^{k}u_{i}^{k} 
  & \leq
  d^{\T}\sigma(x^k, \omega^k)d
  +
  \sum_{i \in I_{0}(x^*)} \theta_{i}^{k}\norm{u_{i}^{k}}^2 \nonumber \\
  & \quad +
  \sum_{i \in I_{B}(x^*)} \gamma_{i}^{k}(u^{k}_{i})^{\T}(\Gamma_{i}\tilde{g}_{i}(x^k))(\Gamma_{i}\tilde{g}_{i}(x^k))^{\T}u^{k}_{i}
  + \delta_{k}\norm{d}^2
\end{align}
for appropriate parameters. First, from the definition of $\Gamma_{i}$, $\tilde{g}_{i}(x^k)$ and $M_{i}(\xi, w)$, we obtain
\begin{equation}
    \label{eq:gammas}
  (\Gamma_{i}\tilde{g}_{i}(x^k))(\Gamma_{i}\tilde{g}_{i}(x^k))^{\T} =
  \left\{
  \begin{array}{ll}
  0, & \mbox{if } \: \overline{g_{i}(x^k)} = 0, \\
  2\norm{\overline{g_{i}(x^k)}}^2
  M_{i}\left(-1, -\displaystyle{\frac{\overline{g_{i}(x^k)}}{\norm{\overline{g_{i}(x^k)}}}}\right), 
  & \mbox{otherwise}
  \end{array}
  \right.
\end{equation}
for all $i \in I_{B}(x^*)$. Moreover, recalling that $V_i^k \in \partial\Pi_{\mathcal{K}_{i}}(-g_{i}(x^k))$, from Lemma~\ref{lemma:bsub} (a), (c), (e), for all $i = 1, \dots, r$, we have 
\begin{equation}
    \label{eq:subgradients}
    V_{i}^{k} = \left\{
    \begin{array}{ll}
        0, 
        & \mbox{if } g_i(x^k) \in \mbox{int}(\mathcal{K}_i), \\
        \displaystyle{\alpha_{i}^{k}M_{i} \left(
        -1, -\frac{\overline{g_{i}(x^k)}}{\norm{\overline{g_{i}(x^k)}}} \right)}
        \mbox{ with } \alpha_{i}^{k} \in \left[0, 1\right],
        & \mbox{if } g_i(x^k) \in \mathrm{bd}^{+}(\mathcal{K}_{i}), \\
        \displaystyle{M_{i}\left(
        -\frac{[g_{i}(x^k)]_0}{\norm{\overline{g_{i}(x^k)}}},
        -\frac{\overline{g_{i}(x^k)}}{\norm{\overline{g_{i}(x^k)}}} \right)},
        & \mbox{if } g_i(x^k) \in \R^{m_i} \setminus (\mathcal{K}_{i} \cup -\mathcal{K}_{i}).
    \end{array}
    \right.
\end{equation}
Recall that $I_{I}(x^*), I_{0}(x^*)$ and $I_{B}(x^*)$ is a partition of $\{1, \dots, r\}$.
Here, we analyze these three cases separately. 

First, let $i \in I_{I}(x^*)$. 
Then, for sufficiently large $k$, we have $g_{i}(x^k) \in \mbox{int}(\mathcal{K}_{i})$.
Hence, from~\eqref{eq:subgradients}, for sufficiently large $k$ we have
\begin{equation}
  \label{eq:socp_akkt2_case_a}
  (u_{i}^{k})^{\T}V_{i}^{k}u_{i}^{k} = 0 \quad \mbox{for all } i \in I_{I}(x^*).
\end{equation}

Now, let $i \in I_{0}(x^*)$.
Note that the sequence $\{ g_i(x^k) \}$ may converge to $g_i(x^*)$ 
from the interior, boundary or origin of $\mathcal{K}_i$ and also from the
outside of $\mathcal{K}_i$. In any case, from Lemmas~\ref{lemma:bsub} and~\ref{lemma:m_eigenvalue}, $\lambda_{\mathrm{max}}(V_{i}^{k}) \leq 1$ holds, where $\lambda_{\mathrm{max}}(V_{i}^{k})$ 
denotes the largest eigenvalue of~$V_{i}^{k}$.
Therefore, if we define
\begin{equation}
    \label{eq:candidate_theta}
    \theta_{i}^{k} := \rho_k \quad \mbox{for all } i \in I_{0}(x^*),
\end{equation}
we obtain
\begin{equation}
  \label{eq:socp_akkt2_case_b}
  \rho_{k}(u_{i}^{k})^{\T}V_{i}^{k}u_{i}^{k} \leq
  \theta_{i}^{k}\norm{u_{i}^{k}}^2 \quad \mbox{for all } i \in I_{0}(x^*).
\end{equation}

From \eqref{eq:socp_akkt2_target_2} and the above two cases, it remains to show the inequality below:
\begin{align*}
    \sum_{i \in I_B(x^*)} \Bigg(
     - \rho_{k}(u_{i}^{k})^{\T}V_{i}^{k}u_{i}^{k}
    + \gamma_{i}^{k}(u^{k}_{i})^{\T}(\Gamma_{i}\tilde{g}_{i}(x^k))
      (\Gamma_{i}\tilde{g}_{i}(x^k))^{\T}u^{k}_{i} \Bigg) & \\
    + d^{\T}\sigma(x^k, \omega^k)d
    & + \delta_{k}\norm{d}^2 \ge 0.
\end{align*}
For all $k$, let us define
\begin{equation}
    \label{eq:suggested_delta}
    \delta_k := \sum_{i \in I_B(x^*)} \varphi_i^k \lambda_{\max} \big(Dg_i(x^k)^\T Dg_i(x^k)\big)
\end{equation}
for some $\varphi_i^k \in \R_+$, where once again
$\lambda_{\mathrm{max}}(Dg_{i}(x^k)^{\T}Dg_{i}(x^k))$ denotes the largest eigenvalue of 
$Dg_{i}(x^k)^{\T} Dg_{i}(x^k)$. 
Observe that $\lambda_{\mathrm{max}}(Dg_{i}(x^k)^{\T}Dg_{i}(x^k))\norm{d}^2 \geq \norm{u_i^k}^2$.
This fact, and the definition of sigma-term show that it is sufficient to prove that 
the following matrix is positive semidefinite:
\begin{equation}
    \label{eq:socp_akkt2_target_B}
    - \rho_{k} V_{i}^k + \gamma_{i}^{k} (\Gamma_{i}\tilde{g}_{i}(x^k))
      (\Gamma_{i}\tilde{g}_{i}(x^k))^{\T}
    - \frac{[\omega_i^k]_0}{[g_i(x^k)]_0} \Gamma_i
    + \varphi_i^k I_{m_i} \quad \mbox{for all } i \in I_{B}(x^*)
\end{equation}
for appropriate choices of $\{\gamma_i^k\} \subset \R$ and $\{\varphi_i^k\} \subset \R_+$.

Thus, from now on let $i \in I_{B}(x^*)$. 
In this case, we can define $N_1, N_2$ and $N_3$ as the partition of $\mathbb{N}$ which satisfy $\{g_i(x^k)\}_{k \in N_1} \subset \mbox{int}(\mathcal{K}_{i})$, $\{g_i(x^k)\}_{k \in N_2} \subset \mathrm{bd}^{+}(\mathcal{K}_{i})$, and $\{g_i(x^k)\}_{k \in N_3} \subset \R^{m_i} \backslash (\mathcal{K}_{i} \cup -\mathcal{K}_{i})$.
We will show that the parameters can be taken as
\begin{equation}
    \label{eq:suggested_varphi}
    \varphi_i^k := 
    \left\{
    \begin{array}{ll} 
        0, & \mbox{if } k \in N_1 \cup N_2, \\
        \rho_k \displaystyle{\frac{\big(\norm{\overline{g_i(x^k)}} - [g_i(x^k)]_0 \big)^2}{[g_i(x^k)]_0 \; \norm{\overline{g_i(x^k)}}}}, & \mbox{if } k \in N_3,
    \end{array}
    \right.
\end{equation}
\begin{equation}
    \label{eq:suggested_gamma}
    \gamma_i^k := \left\{
    \begin{array}{ll} 
        0, & \mbox{if } k \in N_1, \\
        \rho_k \displaystyle{\frac{[g_i(x^k)]_0}{\norm{\overline{g_i(x^k)}}^3}},
        & \mbox{if } k \in N_2 \cup N_3
    \end{array}
    \right.        
\end{equation}
for all $i \in I_{B}(x^*)$. 
Let us then consider the three independent cases.

\begin{enumerate}[(a)]
\item Assume that $\{g_i(x^k)\}_{k \in N_1} \subset \mathrm{int}(\mathcal{K}_{i})$.
From \eqref{eq:subgradients} we have $V_i^k = 0$ and Lemma~\ref{lemma:projection} gives $[\omega_i^k]_0 = 0$ for all $k \in N_1$. Since $\varphi_i^k = 0$ and $\gamma_i^k = 0$ respectively from~\eqref{eq:suggested_varphi} and \eqref{eq:suggested_gamma}, we conclude that~\eqref{eq:socp_akkt2_target_B} is zero.

\item Assume that $\{g_i(x^k)\}_{k \in N_2} \subset \mathrm{bd}^{+}(\mathcal{K}_{i})$.
Clearly, $\varphi_i^k = 0$ from definition and $[\omega_i^k]_0 = 0$ for all $k \in N_2$
from Lemma~\ref{lemma:projection}. Moreover, 
from~\eqref{eq:gammas}, \eqref{eq:subgradients}, we have
\begin{equation*}
  \label{eq:socp_akkt2_case_c2}
  - \rho_k V_{i}^k + \gamma_{i}^{k} (\Gamma_{i} 
  \tilde{g}_{i}(x^k))(\Gamma_{i}\tilde{g}_{i}(x^k))^{\T}
  = \big( - \rho_k \alpha_i^k + 2 \gamma_{i}^{k} \norm{\overline{g_{i}(x^k)}}^2 \big)
  M_i \left( -1,
    -\frac{\overline{g_{i}(x^k)}}{\norm{\overline{g_{i}(x^k)}}}
  \right)
\end{equation*}
for some $\alpha_i^k \in [0,1]$. Since
$[g_i(x^k)]_0 = \norm{\overline{g_{i}(x^k)}} \ne 0$ in this case, 
\eqref{eq:suggested_gamma} shows that $\gamma_{i}^{k} = \rho_k/(\norm{\overline{g_{i}(x^k)}}^2$) for all $k \in N_2$. This shows that 
$- \rho_k \alpha_i^k + 2 \gamma_{i}^{k} \norm{\overline{g_{i}(x^k)}}^2
= \rho_k(2-\alpha_i^k) > 0$.
Since 
$M_i(-1, \overline{g_{i}(x^k)}/\norm{\overline{g_{i}(x^k)}})$ has eigenvalues $0$ or $1$
from Lemma~\ref{lemma:m_eigenvalue}, the above term is positive semidefinite and the 
conclusion follows.

\item Assume that $\{g_i(x^k)\}_{k \in N_3} \subset \R^{m_i} \backslash (\mathcal{K}_{i} \cup -\mathcal{K}_{i})$. First, note that $i \in I_B(x^*)$ implies
\begin{equation}
    \label{eq:case_c_seq}
    0 < [g_i(x^k)]_0 < \norm{\overline{g_{i}(x^k)}}
\end{equation}  
for large enough~$k$.
Furthermore, from Lemma~\ref{lemma:projection}, for all $k \in N_3$
we have
\begin{equation}
    \label{eq:diff_proj}
    [\omega_i^k]_0 = [\rho_k \Pi_{\mathcal{K}_i}(-g_i(x^k))]_0 = (\rho_k
    (\norm{\overline{g_{i}(x^k)}} - [g_i(x^k)]_0))/2
\end{equation}
and thus
\[
  -\frac{[\omega_i^k]_0}{[g_i(x^k)]_0} = 
  \frac{\rho_k}{2} \left( 1 - \frac{\norm{\overline{g_{i}(x^k)}}}{[g_i(x^k)]_0} \right).
\]
For simplicity, let us now define $w_i^k := g_i(x^k) / \norm{\overline{g_{i}(x^k)}}$. 
Using the above result, as well as \eqref{eq:subgradients}, \eqref{eq:suggested_varphi}, and
\eqref{eq:suggested_gamma}, we can observe that \eqref{eq:socp_akkt2_target_B} can
be written as follows:
\begin{multline*}
    -\rho_k \displaystyle{M_i\left(-[w_i^k]_0,-\overline{w_i^k}\right)}
    + 2\rho_k [w_i^k]_0 \displaystyle{M_i\left(-1,-\overline{w_i^k}\right)} \\
    + \rho_k \displaystyle{\left( [w_i^k]_0 + \frac{1}{[w_i^k]_0} - 2 \right) I_{m_i}
    +\frac{\rho_k}{2} \left(1 -  \frac{1}{[w_i^k]_0} \right)} \Gamma_i.
\end{multline*}
Here, we note that $[w_i^k]_0 > 0$ from~\eqref{eq:case_c_seq}, so the above matrix
is well-defined. Using the definition of $M_i$ given in~\eqref{eq:m_i}, this matrix
is equivalent to
\begin{multline}
    \label{eq:with_rho}
    \frac{\rho_k}{2} \Bigg[ 
    - \left( 
        \begin{array}{cc}
            1 & -\overline{w_i^k} \\
            -\overline{w_i^k} & \big(1-[w_i^k]_0\big) I_{m_i-1} + [w_i^k]_0 \, \overline{w_i^k} \, \overline{w_i^k}^\T
        \end{array} 
    \right)
    + 2 [w_i^k]_0
    \left( 
        \begin{array}{cc}
        1 & -\overline{w_i^k} \\
        -\overline{w_i^k} & \overline{w_i^k} \, \overline{w_i^k}^\T
        \end{array} 
    \right) \\
    + 2 \displaystyle{\left( [w_i^k]_0 + \frac{1}{[w_i^k]_0} - 2 \right) I_{m_i}
    + \left( 1 - \frac{1}{[w_i^k]_0} \right)} \Gamma_i
    \Bigg].
\end{multline}
We will show that the above matrix without $\rho_k/2$ has the same structure as
\[
    P := 
    \left[
    \begin{array}{cc}
         \beta & b^\T \\
         b & C 
    \end{array}
    \right]
    \quad \mbox{where} \quad
    C := \displaystyle{\frac{1}{\beta} bb^\T} + \xi I_{m_i-1}
\]
of Lemma~\ref{lemma:psd_matrix}, with $\xi \ge 0$ and $\beta > 0$.
Here, the entries $b$ and $\beta$ are given by
\[
b := \overline{w_i^k} - 2 [w_i^k]_0 \overline{w_i^k} = (1 - 2 [w_i^k]_0) \overline{w_i^k}.
\]
and
\[
\beta := -1 + 2 [w_i^k]_0 + 2 \left( [w_i^k]_0 + \frac{1}{[w_i^k]_0} - 2 \right)
+ \left( 1 - \frac{1}{[w_i^k]_0} \right) = \frac{(1 - 2 [w_i^k]_0)^2}{[w_i^k]_0} \ge 0,
\]
where the last inequality follows because $[w_i^k]_0 > 0$
for sufficiently large $k$.
Moreover, the entry $C$ can be written as 
\begin{align*}
    C = & \, -(1-[w_i^k]_0) I_{m_i-1} - [w_i^k]_0 \overline{w_i^k} \, \overline{w_i^k}^\T 
    + 2 [w_i^k]_0 \overline{w_i^k} \, \overline{w_i^k}^\T \\
    & \, + 2 \left( [w_i^k]_0 + \frac{1}{[w_i^k]_0} - 2 \right) I_{m_i-1}
    - \left( 1 - \frac{1}{[w_i^k]_0} \right) I_{m_i-1} \\
    = & \, 3 \left( [w_i^k]_0 + \frac{1}{[w_i^k]_0} - 2 \right) I_{m_i-1}
    + [w_i^k]_0 \overline{w_i^k} \, \overline{w_i^k}^\T.
\end{align*}
Since $i \in I_B(x^*)$ implies $[g_i(x^k)]_0 / \norm{\overline{g_i(x^k)}} \to 1$, 
for sufficiently large $k$ we can see that
$[w_i^k]_0 \ne 1/2$, that is, $1 - 2 [w_i^k]_0 \ne 0$. 
Thus, using the formula of $b$, we obtain
\begin{align*}
    C = & \, 3 \left( [w_i^k]_0 + \frac{1}{[w_i^k]_0} - 2 \right) I_{m_i-1}
    + \frac{[w_i^k]_0}{(1 - 2 [w_i^k]_0)^2} b \, b^\T \\
    = & \, 3 \left(\frac{\big(\norm{\overline{g_i(x^k)}} - [g_i(x^k)]_0\big)^2}{[g_i(x^k)]_0 \,\norm{\overline{g_i(x^k)}}} \right) I_{m_i-1}
    + \frac{1}{\beta} b \, b^\T,
\end{align*}
where the second equality holds from the formula of $\beta$ and the definition of $[w_i^k]_0$.
Recalling that~\eqref{eq:case_c_seq} holds for sufficiently large $k$, the first term
of the above expression is nonnegative. We conclude from Lemma~\ref{lemma:psd_matrix} that $P$ is positive semidefinite, as well as~\eqref{eq:with_rho}.
\end{enumerate}

Summarizing the above discussion, we take the parameters $\theta_i^k$ and $\gamma_i^k$ as
\eqref{eq:candidate_theta} and \eqref{eq:suggested_gamma}, respectively.
Also, $\delta_k$ is given by \eqref{eq:suggested_delta}, where $\varphi_i^k$ is defined in~\eqref{eq:suggested_varphi}. \red{Let us first suppose that Assumption~\ref{assum:bounded}(a) holds.} In particular, recalling~\eqref{eq:diff_proj}, we know that
$\{ \rho_k \big(\norm{\overline{g_i(x^k)}} - [g_i(x^k)]_0 \big) \}_{k \in N_3}$ is bounded by this assumption (note that $\tilde{N} = N_3$). Since $i \in I_B(x^*)$ implies $\norm{\overline{g_i(x^k)}} - [g_i(x^k)]_0 \to 0$ and~\eqref{eq:case_c_seq} holds
when $k \in N_3$, we obtain $\varphi_i^k > 0$ and 
$\varphi_i^k \to 0$ if $k \in N_3$. \red{Now, suppose that Assumption~\ref{assum:bounded}(b) holds. In this case, note that recalling~\eqref{eq:diff_proj} yields $\varphi_i^k = (2 [\omega_i^k]_0)^2 / (\rho_k [g_i(x^k)]_0 \norm{\overline{g_i(x^k)}})$ when $k \in N_3$. Thus, once again  $\varphi_i^k > 0$ and $\varphi_i^k \to 0$ holds because of Lemma~\ref{lem:los} and~\eqref{eq:case_c_seq}.} 
Thus, we conclude that 
$\delta_k \in \R_+$ and $\delta_k \to 0$ in this case.
Therefore, with such choices of parameters, the inequality~\eqref{eq:socp_akkt2_target_2}
holds, which completes the proof.
\end{proof}

\begin{remark}
    \label{rem:lemmaobst}
    \red{An ideal sequential optimality does not require additional hypotheses, but here we needed Assumption~\ref{assum:bounded}.}
    The main obstacle of the above proof is when $i \in I_B(x^*)$ and the sequence converges from outside of the cone (case (c)). As we already stated in the introduction, in \cite[Theorem~7]{EL20}, the proof is simpler because WCR is assumed, and thus the equality $\overline{g_i(x^k)}^\T \overline{u_i^k} = \norm{\overline{g_i(x^k)}} [u^k_i]_0$ holds. In this case, a simpler formula for~$\gamma_i^k$ would be enough to prove the theorem, and we also would not need to think in the matrix structure of Lemma~\ref{lemma:psd_matrix}. 
    Also in \cite[Theorem~7]{EL20}, Robinson CQ, which implies the boundedness of the sequence of multipliers $\{\omega_i^k\}$, is assumed.
    Here, we replaced this condition with Assumption~\ref{assum:bounded}, \red{where both (a) and (b) are} less restrictive. In fact, based on many calculations, we believe the term $\varphi_i^k$ with
    $i \in I_B(x^*)$ and $k \in N_3$ should depend on~$\rho_k$. Since $\rho_k \to \infty$, it appears to be difficult to guarantee that $\varphi_i^k \to 0$ without any assumption. 
    
    \red{At first, assumption (b) seems more reasonable than (a) because it only involves the problem's data. Also, from Lemma~\ref{lem:los}, we cannot guarantee the boundedness of the multipliers. In fact, the limit in Lemma~\ref{lem:los} allows the multipliers $\omega^k_i$ and $\mu^k_i$ to be not bounded. Since (a) requires the first-term of the multiplier $\omega^k_i$ to be bounded when the sequence converges from outside of the cone ($k \in \tilde{N}$), (b) is weaker than (a) in this sense. However, Lemma~\ref{lem:los} also says that (b) imposes a condition on the multiplier $\mu_i^k$, associated to the equality constraints, while (a) does not. For this reason, assumption (b) is not necessarily weaker than (a). Moreover, whether assumption~(a) is strong seems to depend on the specific implementation of the algorithms. In particular, an algorithm that generates only feasible points would have $\tilde{N} = \emptyset$, and thus (a) would not need to be considered.}
\end{remark}


\subsection{Validity of the new optimality condition}
The following theorem shows that AKKT2 is in fact a second-order optimality condition 
that does not require any constraint qualification.
\begin{theorem}
  \label{theorem:akkt2}
  Let $x^*$ be a local \red{minimizer} of \eqref{eq:socp} \red{and suppose that Assumption~\ref{assum:bounded} holds.} Then, $x^*$ satisfies AKKT2.
\end{theorem}
\noindent
\begin{proof} First, we show that the AKKT conditions hold at $x^*$.
Let $\{\rho_k\}$ be a sequence such that $\rho_k \to \infty$.
From Theorem \ref{theorem:akkt}, there exists $\{x^k\}$, where $x^k$ is a local minimizer of $F_k$ such that $x^k \to x^*$. As in the theorem, let $\mu^k := \rho_k h(x^k)$ and $\omega^k_i := \rho_k\Pi_{\mathcal{K}_i}(-g_i(x^k))$ for all~$k$. Then, $(x^k, \mu^k, \omega^k)$ \red{satisfies} the AKKT conditions.

Next, we show that $x^*$ satisfies the second-order condition.
The definition of $F_k$ in~\eqref{eq:penalty_function} 
and the local optimality of $x^k$ give
\begin{multline*}
  \grad F_k(x^k) :=
  \grad f(x^k) +
  \norm{x^k - x^*}^2 (x^k - x^*) -
  \rho_k\sum_{i=1}^{r}Dg_i(x^k)^{\T}\Pi_{\mathcal{K}_i}(-g_i(x^k)) \\ +
  \rho_k \grad h(x^k)^{\T}h(x^k) = 0.
\end{multline*}
Moreover, from Theorem \ref{theorem:generalized_hessian}, there exists $\chi_i^k \in \partial(\Pi_{\mathcal{K}_{i}} \circ (-g_i))(x^k)$ such that $d^{\T} \grad^2 F_k(x^k)d \geq 0$ where $\grad^2 F_k(x^k)$ is an element of the generalized Hessian of $F_k$ at $x_k$. Namely,
\begin{align*}
  \grad^2 F_k(x^k) & :=
  \grad_{x}^{2} L(x^k, \mu^k, \omega^k) +
  \norm{x^k - x^*}^{2} I_{n} +
  2(x^k-x^*)(x^k-x^*)^{\T} \\
  & \ +
  \rho_{k} \sum_{i=1}^{p}\grad h_{i}(x^k)\grad h_{i}(x^k)^{\T} -
  \rho_{k} \sum_{i=1}^{r} Dg_{i}(x^k)^{\T} \chi_i^k.
\end{align*}
From Theorem \ref{theorem:convolution}, there exists $V_{i}^{k} \in \partial\Pi_{\mathcal{K}_{i}}(-g_{i}(x^k))$ such that $\chi_i^k = V_{i}^{k} \circ -Dg_{i}(x^k) = - V_{i}^{k}Dg_{i}(x^k)$ which \red{satisfies} the following inequality for each $d \in \R^n$:
\begin{align*}
  d^{\T} \left(
      \grad_{x}^{2} L(x^k, \mu^k, \omega^k)
    + \rho_{k} \sum_{i=1}^{p}\grad h_{i}(x^k)\grad h_{i}(x^k)^{\T} \right. \\
    \left. + \rho_{k} \sum_{i=1}^{r} Dg_{i}(x^k)^{\T}V_{i}^{k}Dg_{i}(x^k)
    + \Omega_k
  \right)d
  \geq
  0,
\end{align*}
where $\Omega_k := \norm{x^k - x^*}^{2}I_n + 2(x^k-x^*)(x^k-x^*)^{\T} \to \red{0}$. Defining $\eta^k_i := \rho_k$ for all $i$, from Lemma~\ref{lemma:akkt2_inequality}, we obtain
\begin{multline*}
    d^{\T}\left(
    \grad^{2}_{x}L(x^k, \mu^k, \omega^k)
    + \sigma(x^k, \omega^k)
    + \sum_{i=1}^{p}\eta^{k}_{i}\grad h_i(x^k)\grad h_i(x^k)^{\T} \right. \\
    + \sum_{i \in I_{B}(x^*)}
    \gamma_i^k
    (Dg_{i}(x^k)^{\T}\Gamma_{i}\tilde{g}_{i}(x^k))
    (Dg_{i}(x^k)^{\T}\Gamma_{i}\tilde{g}_{i}(x^k))^{\T} \\
    + \sum_{i \in I_{0}(x^*)} \theta_{i}^{k}Dg_{i}(x^k)^{\T}Dg_{i}(x^k)    
    + \tilde{\delta}_{k} I_{n} + \Omega_k \Bigg)d \geq 0,
\end{multline*}
for all $d \in \R^n$ with appropriate parameters $\{\theta_{i}^{k}\} \subset \R, \{\gamma_{i}^{k}\} \subset \R$ and $\{\tilde{\delta}_{k}\} \subset \R_{+}$ with $\tilde{\delta}_k \to 0$ for sufficiently large $k$. Taking 
$\delta_k := \tilde{\delta}_k + 3 \norm{x^k - x^*}^2$ for all $k$, we observe that
$\{\delta_k\} \subset \R_+$, $\delta_k \to 0$ and
\[
  \delta_k \norm{d}^2
  \ge \tilde{\delta}_k \norm{d}^2 + \norm{x^k - x^*}^2 \norm{d}^2 +
  2 (d^\T (x^k - x^*))^2 
  =  d^\T (\tilde{\delta}_k I_n + \Omega_k) d,
\]
where the inequality holds from Cauchy–Schwarz inequality. This, together with the previous inequality yields the second-order condition. 
Therefore, $x^*$~is an AKKT2 point.
\end{proof}

\red{From Theorem~\ref{theorem:akkt}, the same result above can be applied by replacing AKKT2 with CAKKT2.}

\subsection{Strength of the new optimality condition}
In this section, we show the relationship between AKKT2 and WSONC. The following result shows that AKKT2 is a strong optimality condition in the sense that it implies WSONC under a constraint qualification. This extends the result for NLP~\cite[Proposition 3.5]{AN17}.
\begin{theorem}
  If $x^*$ is an AKKT2 point of \eqref{eq:socp} and fulfills Robinson CQ and WCR, then $x^*$ satisfies WSONC with some multipliers $\mu^*$ and $\omega^*$.
\end{theorem}

\begin{proof}
Since $x^*$ fulfills AKKT2, there exist sequences $\R^n \supset \{x^k\} \to x^*$, $\{\mu^k\} \subset \R^p$, $\{\omega^k\} \subset \mathcal{K}$, $\{\eta^k\} \subset \R^p$, $\{\theta_{i}^k\}_{i \in I_{0}(x^*)} \subset \R$, $\{\gamma_{i}\}_{i \in I_{B}(x^*)} \subset \R$, and $\R_{+} \supset \{\delta^k\} \to 0$ such that
\begin{equation*}
  \grad_x L(x^k, \mu^k, \omega^k) \to 0, 
\end{equation*}
and
\begin{align*}
    d^{\T}\Bigg(
    \grad^{2}_{x}L(x^k, \mu^k, \omega^k)
    + \sigma(x^k, \omega^k)
    + \sum_{i=1}^{p}\eta^{k}_{i}\grad h_i(x^k)\grad h_i(x^k)^{\T} \\
    + \sum_{i \in I_{B}(x^*)}
    \gamma_i
    (Dg_{i}(x^k)^{\T}\Gamma_{i}\tilde{g}_{i}(x^k))
    (Dg_{i}(x^k)^{\T}\Gamma_{i}\tilde{g}_{i}(x^k))^{\T}\\
    + \sum_{i \in I_{0}(x^*)} \theta_{i}^{k}Dg_{i}(x^k)^{\T}Dg_{i}(x^k) 
    \Bigg)d \geq -\delta_{k}\norm{d}^2
\end{align*}
for all $d \in \R^n$ and sufficiently large $k$.
By~\cite[Theorem 3.3]{AN19}, there exist subsequences of $\{\mu^k\} \to \mu^*$ and $\{\omega^k\} \to \omega^*$ where $(x^*, \mu^*, \omega^*)$ is a KKT point under Robinson CQ. We consider these subsequences below.
Moreover, by Lemma~\ref{lemma:inner_semicontinuity}, \red{for each $d \in S(x^*)$}, we can take a sequence $\{d^k\} \subset S(x^k, x^*)$ which fulfills $d^k \to d$ where
\begin{equation*}
  S(x^k, x^*)
  :=
  \left\{
    d \in \R^n \middle|
    \begin{array}{l}
      Dh(x^k)d = 0; \: Dg_{i}(x^k)d = 0, i \in I_{0}(x^*); \\
      \tilde{g}_{i}(x^k)^{\T}\Gamma_{i}Dg_{i}(x^k)d = 0, i \in I_{B}(x^*)
    \end{array}
  \right\}.
\end{equation*}
From the definition of $S(x^k, x^*)$, we obtain
\begin{equation*}
  (d^k)^{\T}
  (
    \grad^{2}_{x}L(x^k, \mu^k, \omega^k)
    + \sigma(x^k, \omega^k)
  )
  d^k
  \geq
  -\delta_k\norm{d^k}^2,
\end{equation*}
for sufficiently large $k$. Taking the limit $k \to \infty$ in the above inequality, we have 
\[
d^{\T}(\grad^{2}_{x}L(x^*, \mu^*, \omega^*) + \sigma(x^*, \omega^*))d \geq 0,
\]
which means WSONC.
\end{proof}
\red{As we stated before, the above result holds when AKKT2 is replaced by CAKKT2.}

\subsection{Numerical example}

By the definition of the AKKT2 conditions, it is obvious that AKKT2 implies AKKT. We now give an example that shows that AKKT2 is in fact stronger than AKKT. The example is similar to~\cite[Example 3.1]{AN17}. Consider the \eqref{eq:socp} without equality constraints where $n := 2$, $r := 1$, $m_1 := 2$, $f(x) := -x_1 - x_2$ and $g(x) := (1, x_1^2 x_2^2)^{\T}$. Let $x^* := (1, 1)^{\T}$. Since the Lagrangian of the problem is defined by
\begin{equation*}
  \grad_x L(x, \omega) =
  \begin{pmatrix}
    -1 \\
    -1
  \end{pmatrix}
  - 2 x_1 x_2
  \begin{pmatrix}
    0 & x_2 \\
    0 & x_1
  \end{pmatrix}
  \omega,
\end{equation*}
$x^*$ is an AKKT (and KKT) point with Lagrange multiplier $\omega^* = (1/2, -1/2)^{\T}$. Now, we will show that the AKKT2 conditions fail at $x^*$. Suppose that AKKT2 holds and let $x^k := (x_1^k, x_2^k)^{\T} \to x^*$ and $d^k := (x_1^k, -x_2^k)^{\T}$. Then, from the definition of AKKT2, we obtain
\begin{align*}
  (d^k)^{\T}\grad_x^2 L(x^k, \omega^k)d^k + \delta_k \norm{d^k}^2 & =
  4 \omega_2^k (x_1^k x_2^k)^2 + \delta_k ((x_1^k)^2 + (x_2^k)^2) \geq 0.
\end{align*}
However, since $x^k \to x^*$, $\omega^k \to \omega^*$ and $\delta_k \to 0$ hold, we have
\[
4 \omega_2^k (x_1^k x_2^k)^2 + \delta_k ((x_1^k)^2 + (x_2^k)^2) \to -2,
\]
which contradicts the result. Consequently, $x^*$ is an AKKT point but not an AKKT2 point.

\section{Algorithms that generate AKKT2 sequences}
\label{sec:algorithms}

As stated previously, having a formal definition of AKKT2 is important because now
we can propose and analyze algorithms. Thus, in this section, we will consider algorithms that generate AKKT2 sequences. We consider two algorithms: \red{an} augmented Lagrangian method and \red{a} sequential quadratic programming (SQP) method.

\subsection{The augmented Lagrangian method}
Our approach is based on an existing augmented Lagrangian method for SOCP (see for instance~\cite{LI07}) and for nonlinear programming, that is proposed in~\cite[Algorithm~4.1]{AN07} and~\cite[Algorithm~1]{AN17}. Note that in~\cite{AN17} the method is shown to generate AKKT2 sequences.
The augmented Lagrangian function used here is defined as
\begin{equation*}
  L_{\rho}(x, \mu, \omega) :=
  f(x) +
  \frac{1}{2\rho}\left\{
    \norm{\mu - \rho h(x)}^2 +
    \sum_{i=1}^{r} \left(
      \norm{\Pi_{\mathcal{K}_i}(\omega_{i} - \rho g_{i}(x))}^2 - \norm{\omega_{i}}^2
    \right)
  \right\},
\end{equation*}
where $\rho > 0$ is the penalty parameter. The gradient of $L_{\rho}$ with respect to $x$ is given by
\begin{equation*}
  \grad_{x} L_{\rho}(x, \mu, \omega) =
  \grad f(x) -
  Dh(x)^{\T} (\mu - \rho h(x)) - 
  \sum_{i=1}^{r} Dg_{i}(x)^{\T} \Pi_{\mathcal{K}_i}(\omega_{i} - \rho g_{i}(x)).
\end{equation*}
To handle second-order information, we need to obtain the second-order derivative of~$L_{\rho}$. However, as it can be seen above, its derivative is not differentiable because of the projection term. To overcome such difficulty, we consider the following.

As an alternative for the second-order derivative of~$L_{\rho}$, we define an operator $\bar{\grad}^{2}_{x}$ as follows:
\begin{align*}
  \bar{\grad}^{2}_{x} L_{\rho}(x, \mu, \omega)  & :=
  \grad_{x}^{2} L(x, \mu - \rho h(x), \Pi_{\mathcal{K}}(\omega - \rho g(x))) \\
  & + \rho \sum_{i=1}^{p}\grad h_{i}(x)\grad h_{i}(x)^{\T}
  + \rho \sum_{i=1}^{r} Dg_{i}(x)^{\T}V_{i}Dg_{i}(x),
\end{align*}
where
\begin{equation*}
  V_{i} =
  \left\{
    \begin{array}{ll}
    I_{m_i}, &
    \mathrm{if\ } -g_{i}(x) \in \mathrm{bd}(\mathcal{K}_{i}) \cup \mathrm{int}(\mathcal{K}_{i}), \\
    0, &
    \mathrm{if\ } -g_{i}(x) \in \mathrm{int}(-\mathcal{K}_{i}), \\
    \displaystyle{M_{i}\left(
      -\frac{[g_{i}(x)]_{0}}{\norm{\overline{g_{i}(x)}}},
      -\frac{\overline{g_{i}(x)}}{\norm{\overline{g_{i}(x)}}}
    \right),} &
    \mathrm{otherwise,}
    \end{array}
  \right.
\end{equation*}
recalling that $M_i$ is defined in~\eqref{eq:m_i}. From Lemma~\ref{lemma:bsub}, note that $V_i \in \partial \Pi_{\mathcal{K}_i}(-g_i(x))$. The augmented Lagrangian method that uses the above defined $\bar\nabla_x^2$ is given in Algorithm~\ref{alg:socp_akkt2}. \red{This is similar to the NLP case~\cite{AN17}, where Step~2 corresponds to the subproblem, Step~3 defines the penalty parameter $\rho_k$ and Steps~4 and~5 updates the multipliers $\mu^k$ and $\omega^k$, using safeguarded multipliers $\hat{\mu}^k$ and $\hat{\omega}^k$. For practical details, see~\cite{AHMRS22,BM14}.}

\begin{algorithm}[htb]
  \caption{Augmented Lagrangian method that generates AKKT2 sequences}
  \label{alg:socp_akkt2}
  \begin{algorithmic}[1]
  \State Let $\gamma > 1, \rho_{1} > 0$ and $\tau \in (0, 1)$. Define a sequence $\{\epsilon_k\} \subset \R_{+}$ such that $\epsilon_k \to 0$. Set an initial point $x_0 \in \R^n, \hat{\mu}_i^1 \in \R$ for all $i \in \{1, \dots, p\}$ and $\hat{\omega}_j^1$ such that $\hat{\omega}_j^1 \in \mathcal{K}_j$ for all $j \in \{1, \dots, r\}$. Initialize $k := 1$ and $\norm{v^0} := +\infty$.
  \State Find an appropriate minimizer $x^k$ of $L_{\rho_k}(x, \hat{\mu}^k, \hat{\omega}^k)$, which satisfies
  \begin{equation}
    \label{eq:algorithm_target}
    \norm{\grad_x L_{\rho_k}(x^k, \hat{\mu}^k, \hat{\omega}^k)} \leq \epsilon_k \quad \mathrm{and} \quad 
    \bar{\grad}^2_x L_{\rho_k}(x^k, \hat{\mu}^k, \hat{\omega}^k) + \red{\epsilon_k I_n \succeq 0}.
  \end{equation}
  \State Define $v^k = (v_1^k, \dots, v_r^k)$ with
  \begin{equation*}
    v_i^k := \Pi_{\mathcal{K}_i}\left(\frac{\hat{\omega}_i^k}{\rho_k} - g_i(x^k)\right) - \frac{\hat{\omega}_i^k}{\rho_k},
  \end{equation*}
  for all $i \in \{1, \dots, r\}$.
  If $\mathrm{max}\{\norm{h(x^k)}_{\infty}, \norm{v^k}_{\infty}\} \leq \tau \mathrm{max}\{\norm{h(x^{k-1})}_{\infty}, \norm{v^{k-1}}_{\infty}\}$ holds, set $\rho_{k+1} = \rho_k$, otherwise, set $\rho_{k+1} = \gamma \rho_k$.
  \State Compute $\mu^k := \hat{\mu}^k - \rho_k h(x^k)$ and $\omega_i^k := \Pi_{\mathcal{K}_i}(\hat{\omega}_i^k - \rho_k g_i(x^k))$ for all $i \in \{1, \dots, r\}$.
  \State Define $\hat{\mu}_i^{k+1}$ and $\hat{\omega}_j^{k+1}$ such that $\hat{\omega}_j^{k+1} \in \mathcal{K}_j$ and they are \red{uniformly} bounded for all $i \in \{1, \dots, p\}$ and $j \in \{1, \dots, r\}$.
  \State Set $k := k+1$ and go to Step 2.
  \end{algorithmic}
\end{algorithm}

\begin{theorem}
  \red{Let $\{x^k\}$ be generated by Algorithm \ref{alg:socp_akkt2}. Then, any accumulation point of $\{ x^k \}$ satisfies the AKKT2 conditions.} 
\end{theorem}

\begin{proof}  
  Let $x^*$ be an accumulation point of $\{ x^k \}$ and assume that $\{\rho_k\}$ is unbounded. From simple calculations, $\norm{\grad_x L_{\rho_k}(x^k, \hat{\mu}^k, \hat{\omega}^k)} \leq \epsilon_k$ implies $\norm{\grad_x L(x^k, \mu^k, \omega^k)} \leq \epsilon_k$ where $\mu^k := \hat{\mu}^k - \rho_k h(x^k)$ and $\omega_i^k := \Pi_{\mathcal{K}_i}(\hat{\omega}_i^k - \rho_k g_i(x^k))$. In addition, from Lemma~\ref{lemma:akkt2_inequality}, we obtain
\begin{multline*}
  d^{\T}\left(\sum_{i=1}^{r} \rho_{k}Dg_{i}(x^k)^{\T}V_{i}^{k}Dg_{i}(x^k)\right)d
  \leq
  d^{\T}\left(\sigma(x^k, \omega^k)
  +
  \sum_{i \in I_{0}(x^*)} \theta_{i}^{k} Dg_{i}(x^k)^{\T}Dg_{i}(x^k)\right. \\
  + \left.
  \sum_{i \in I_{B}(x^*)} \gamma_{i}^{k} Dg_{i}(x^k)^{\T}\Gamma_{i}\tilde{g}_{i}(x^k)(Dg_{i}(x^k)^{\T}\Gamma_{i}\tilde{g}_{i}(x^k))^{\T}
  + \delta_{k} I_n \right)d,
\end{multline*}
for all $d \in \R^n$ for sufficiently large $k$ and appropriate parameters $\{\theta_i^k\} \subset \R, \{\gamma_i^k\} \subset \R$ and $\R_{+} \supset \{\delta_i^k\} \to 0$. The above inequality, together with $\bar{\grad}^2_x L_{\rho_k}(x^k, \hat{\mu}^k, \hat{\omega}^k) + \red{\epsilon_k I_n \succeq 0}$, indicates
\begin{multline*}
    d^{\T}\Bigg(
    \grad^{2}_{x}L(x^k, \mu^k, \omega^k)
    + \sigma(x^k, \omega^k)
    + \sum_{i=1}^{p}\rho_k\grad h_i(x^k)\grad h_i(x^k)^{\T} \\
    + \sum_{i \in I_{B}(x^*)}
    \gamma_i
    (Dg_{i}(x^k)^{\T}\Gamma_{i}\tilde{g}_{i}(x^k))
    (Dg_{i}(x^k)^{\T}\Gamma_{i}\tilde{g}_{i}(x^k))^{\T} \\
    + \sum_{i \in I_{0}(x^*)} \theta_{i}^{k}Dg_{i}(x^k)^{\T}Dg_{i}(x^k) 
    + (\delta_k + \epsilon_k)I_n \Bigg)d \geq 0
\end{multline*}
for all $d \in \R^n$, which shows that $x^*$ is an AKKT2 point if~\eqref{eq:algorithm_target} holds. 

Assuming that $\{\rho_k\}$ is bounded, we can guarantee that $x^*$ is an AKKT point from~\cite[Theorem 5.2]{AN19}. Moreover, the second-order information can be proved using the same tools of Lemma~\ref{lemma:akkt2_inequality}. In fact, the inequalities that appear in that lemma are not affected by $\{ \rho_k \}$ being bounded. Therefore, we conclude that Algorithm~\ref{alg:socp_akkt2} generates AKKT2 sequences.
\end{proof}

\subsection{The sequential quadratic programming method}
Let us now consider an SQP method that generates AKKT2 points.
SQP methods are \red{iterative methods where, at each iteration, the search direction is determined by solving subproblems in which a quadratic model of the objective function is minimized} \red{subject to linearized constraints. To ensure global convergence, a penalty function is required. To address} ill-posed problems and generate well-defined and practical subproblems, we consider an algorithm based on a stabilized SQP method. This approach was first proposed for NLP in~\cite{WR98} and later extended to nonlinear semidefinite programming~\cite{YA22}, which includes SOCP.

In the following, we describe the outline of the algorithm. Let $k \in \mathbb{N}$ be the current iteration and $H_k \in \R^{n \times n}$ be the Hessian of the Lagrangian function or its approximation. For the current point $(x^k, \mu^k, \omega^k) \in \R^n \times \R^p \times \mathcal{K}$, the algorithm solves the following subproblem:
\begin{equation}
  \label{eq:sqp_subproblem}
  \begin{array}{ll}
    \mini{(\xi, \chi) \in \R^n \times \mathcal{K}} & \inner{\grad f(x^k) - Dh(x^k)s^k}{\xi} + \frac{1}{2} \inner{M_k \xi}{\xi} + \frac{\rho_k}{2} \norm{\chi}^2 \\[5pt]
    \mathrm{subject\ to\ } & Dg(x^k)\xi + \rho_k (\chi - t^k) \in \mathcal{K},
  \end{array}
\end{equation}
where $\rho_k > 0$ is the penalty parameter and
\begin{equation*}
  s^k := \mu^k - \frac{1}{\rho_k}h(x^k), \quad t^k := \omega^k - \frac{1}{\rho_k}g(x^k), \quad M_k := H_k + \frac{1}{\rho_k} \grad h(x^k) \grad h(x^k)^{\T}.
\end{equation*}
Note that (\ref{eq:sqp_subproblem}) always has a strictly feasible point $(\xi, \chi) = (0, t^k + (1, 0, \dots, 0)^{\T})$, which indicates that it fulfills Slater's constraint qualification. Moreover, if $M_k \succ O$ for each iteration, (\ref{eq:sqp_subproblem}) has a unique global minimizer. Using the solution of the subproblem $(\xi, \chi)$, we define the search direction $p^k$ and the candidates of Lagrange multipliers $(\bar{\mu}^{k+1}, \bar{\omega}^{k+1})$ as
\begin{equation*}
  p^k := \xi, \quad \bar{\mu}^{k+1} := \mu^k - \frac{1}{\rho_k}(h(x^k) + Dh(x^k)\xi), \quad \bar{\omega}^{k+1} := \chi.
\end{equation*}

Instead of choosing $(\bar{\mu}^{k+1}, \bar{\omega}^{k+1})$ as the new Lagrange multipliers immediately, we check their optimality by the so-called VOMF procedure~\cite{G113}. In this procedure, each iterate is classified into four groups, V-, O-, M- and F- iterates and we update the Lagrange multipliers in different ways, which enables global convergence.

To introduce VOMF procedure, we give some notations.
We extend the concept of $\epsilon$--active set for nonlinear programming and denote the sets of indices by
\begin{align*}
  I_{0, \epsilon}(x) & := \{i \in \{1, \dots, r\} \mid \norm{g_i(x)} \leq \epsilon\}, \\
  I_{B, \epsilon}(x) & := \{i \in \{1, \dots, r\} \mid |[g_i(x)]_0 - \norm{\overline{g_i(x)}}| \leq \epsilon, [g_i(x)]_0 > 0 \}, \\
  I_{BB, \epsilon}(x, \omega) & := \{i \in I_{B, \epsilon}(x) \mid |[\omega_i]_0 - \norm{\overline{\omega_i}}| \leq \epsilon, [\omega_i]_0 > 0\}.
\end{align*}
We define $r_V \colon \R^n \to \R_{+}$ and $r_O \colon \R^n \times \R^p \times \mathcal{K} \times \R_{+} \times \R_{+} \to \R_{+}$ with
\begin{align*}
  r_V(x) & := \norm{h(x)} + \norm{\Pi_{\mathcal{K}}(-g(x))}, \\
  r_O(x, \mu, \omega, \rho, \epsilon) & := \norm{\grad_x L(x, \mu, \omega)} + \red{\sum_{i=1}^r |g_i(x) \circ \omega_i|} + \max \big\{ \lambda_{\mathrm{max}}(-P(x, \mu, \omega, \rho, \epsilon)), 0 \big\},
\end{align*}
\red{where $g_i(x) \circ \omega_i$ is the Jordan product of $g_i(x)$ and $\omega_i$,} and
$\lambda_{\mathrm{max}}(-P(x, \mu, \omega, \rho, \epsilon))$ is the largest eigenvalue of \red{$- P(x, \mu, \omega, \rho, \epsilon)$, with}
\begin{align*}
  & P(x, \mu, \omega, \rho, \epsilon) := \\ & 
  \grad^{2}_{x}L(x, \mu, \omega) -
  \sum_{i \in I_{BB, \epsilon}(x, \omega)} \frac{[\omega_i]_0}{[g_i(x)]_0} Dg_{i}(x)^{\T}\Gamma_{i}Dg_{i}(x) +
  \sum_{i=1}^{p}\frac{1}{\rho}\grad h_i(x)\grad h_i(x)^{\T} \\
  & +
  \sum_{i \in I_{0, \epsilon}(x)} \frac{1}{\rho} Dg_{i}(x)^{\T}Dg_{i}(x) +
  \sum_{i \in I_{B, \epsilon}(x)} \frac{1}{\rho}
  (Dg_{i}(x)^{\T}\Gamma_{i}\tilde{g}_{i}(x))
  (Dg_{i}(x)^{\T}\Gamma_{i}\tilde{g}_{i}(x))^{\T}.
\end{align*}

First, we note that when $i \in I_B(x^*)$, the associate multiplier $\omega^*_i$ should be $0$ or in the boundary of $\mathcal{K}_i$ because of the complementarity condition. Thus, the sigma-term~\eqref{eq:sigma-term} at such $(x^*,\omega^*)$ only takes $i$ in $\{ i \in I_B(x^*) \mid \omega_i^* \in \mathrm{bd}^+(\mathcal{K}_i)\}$, which justifies the first summation of the above definition.
Furthermore, $r_V$ and $r_O$ indicate the violation of the constraints and the optimality measure with respect to the AKKT2 conditions, respectively. \red{The complementarity term in $r_O$ is actually related to CAKKT conditions, but CAKKT implies AKKT in SOCP~\cite[Theorem~4.2]{AN09}. One can also use $\inner{g(x)}{\omega}$ as the measure of complementarity, but this would be related to the so-called TAKKT. It is known, however, that TAKKT and AKKT are independent conditions.} We also define $\Phi(x, \mu, \omega, \rho, \epsilon) := r_V(x) + \kappa r_O(x, \mu, \omega, \rho, \epsilon)$ and $\Psi(x, \mu, \omega, \rho, \epsilon) := \kappa r_V(x) + r_O(x, \mu, \omega, \rho, \epsilon)$ where $\kappa \in (0, 1)$ is a weight parameter. Moreover, the merit function of the problem is denoted by
\begin{equation*}
  \tilde{F}(x; \rho, \mu, \omega) := f(x) + \frac{1}{2\rho} \norm{\rho\mu - h(x)}^2 + \frac{1}{2\rho} \norm{\Pi_{\mathcal{K}}(\rho\omega - g(x))}^2,
\end{equation*}
where $\rho > 0$ is a parameter. The gradient of $\tilde{F}$ with respect to $x$ is given by
\begin{equation*}
  \grad \tilde{F}(x; \rho, \mu, \omega) := \grad f(x) - Dh(x)^{\T}\left(\mu - \frac{1}{\rho}\right) + Dg(x)^{\T}\Pi_{\mathcal{K}}\left(\omega - \frac{1}{\rho}g(x)\right).
\end{equation*}
Using the above notations, constants $\mu_{\mathrm{max}} > 0,\ \omega_{\mathrm{max}} > 0$ and positive parameters $\phi_k, \psi_k, \gamma_k$ and $\epsilon_k$, we introduce VOMF procedure in Algorithm \ref{alg:vomf}.
\begin{algorithm}[htb]
  \caption{VOMF procedure}
  \label{alg:vomf}
  \begin{algorithmic}[1]
  \If{$\Phi(x^{k+1}, \bar{\mu}^{k+1}, \bar{\omega}^{k+1}, \rho_k, \epsilon_k) \leq \frac{1}{2}\phi_k$}
    \State (V-iterate) Set
      \begin{equation*}
        \mu^{k+1} := \bar{\mu}^{k+1},\ \omega^{k+1} := \bar{\omega}^{k+1},\ \phi_{k+1} := \frac{1}{2}\phi_k,\ \psi_{k+1} := \psi_k,\ \gamma_{k+1} := \gamma_k.
      \end{equation*}
  \ElsIf{$\Psi(x^{k+1}, \bar{\mu}^{k+1}, \bar{\omega}^{k+1}, \rho_k, \epsilon_k) \leq \frac{1}{2}\psi_k$}
    \State (O-iterate) Set
    \begin{equation*}
      \mu^{k+1} := \bar{\mu}^{k+1},\ \omega^{k+1} := \bar{\omega}^{k+1},\ \phi_{k+1} := \phi_k,\ \psi_{k+1} := \frac{1}{2}\psi_k,\ \gamma_{k+1} := \gamma_k.
    \end{equation*}
  \ElsIf{$\norm{\grad F(x^{k+1}; \rho_k, \mu^k, \omega^k)} \leq \gamma_k$}
    \State (M-iterate) Set
      \begin{align*}
        & \mu^{k+1} := \Pi_{C}\left(\mu^k - \frac{1}{\rho_k}h(x^{k+1})\right),
        \omega^{k+1} := \Pi_{D}\left(\omega^k - \frac{1}{\rho_k}g(x^{k+1})\right), \\
        & \phi_{k+1} := \phi_k,\ \psi_{k+1} := \psi_k,\ \gamma_{k+1} := \frac{1}{2}\gamma_k,
      \end{align*}
    where
    \begin{align*}
      & C := \{\mu \in \R^p \mid -\mu_{\mathrm{max}}e \leq \mu \leq \mu_{\mathrm{max}}e\}, \\
      & D := \{\omega = (\omega_1, \dots, \omega_r) \in \mathcal{K} \mid -\omega_{\mathrm{max}}e \leq \omega_i \leq \omega_{\mathrm{max}}e, i = 1, \dots, r\}.
    \end{align*}
  \Else
    \State (F-iterate) Set
      \begin{equation*}
        \mu_{k+1} := \mu_k,\ \omega_{k+1} := \omega_k,\ \phi_{k+1} := \phi_k,\ \psi_{k+1} := \psi_k,\ \gamma_{k+1} := \gamma_k.
      \end{equation*}
  \EndIf
  \end{algorithmic}
\end{algorithm}
If an iterate is classified as V- or O-iterate, we consider that $(x^{k+1}, \bar{\mu}^{k+1}, \bar{\omega}^{k+1})$ is approaching an AKKT2 point, therefore, we set $(\bar{\mu}^{k+1}, \bar{\omega}^{k+1})$ as the new Lagrange multipliers. If an iterate belongs to M-iterate, we update the Lagrange multipliers similarly to the augmented Lagrangian method since $\tilde{F}$ can be regarded as the augmented Lagrangian function and its value is decreasing. Otherwise, we consider that the candidates of Lagrange multipliers are inappropriate and we do nothing. We give a summary of SQP method in Algorithm \ref{alg:sqp} and then analyze its convergence.

\begin{algorithm}[htb]
  \caption{SQP method that generates AKKT2 sequences}
  \label{alg:sqp}
  \begin{algorithmic}[1]
  \State Set constants $\tau \in (0, 1), \alpha \in (0, 1), \beta \in (0, 1), \kappa \in (0, 1), \mu_{\mathrm{max}} > 0$ and $\omega_{\mathrm{max}} > 0$. Choose an initial point $(x_0, \mu_0, \omega_0)$ and parameters $\phi_0 > 0, \psi_0 > 0, \gamma_0 > 0, \rho_0 > 0, k := 0, \bar{\mu}_0 := \mu_0, \bar{\omega}_0 := \omega_0$ and $\{\epsilon_k\} \subset \R_{+}$ such that $\epsilon_k \to 0$.
  \State If $\norm{\grad F(x_k; \rho_k, \mu_k, \omega_k)} = 0$, set
    \begin{equation*}
      x^{k+1} := x^k,\ \bar{\mu}^{k+1} := \mu^k - \frac{1}{\rho_k}h(x^{k+1}),\ \bar{\omega}^{k+1} := \Pi_{\mathcal{K}}\left(\omega^k - \frac{1}{\rho_k}g(x^{k+1})\right),
    \end{equation*}
    and go to Step 5. Otherwise, go to Step 3.
  \State Choose $M_k \succ O$ and find the solution of (\ref{eq:sqp_subproblem}) $(\xi^k, \chi^k)$. Set
    \begin{equation*}
      p^k := \xi^k,\ \bar{\mu}^{k+1} := \mu^k - \frac{1}{\rho_k}(h(x^k) + Dh(x^k)\xi^k),\ \bar{\omega}^{k+1} := \chi^k.
    \end{equation*}
  \State Compute the smallest nonnegative integer $\ell_k$ which fulfills
    \begin{equation*}
      F(x^k + \beta^{\ell_k}p^k; \rho_k, \mu^k, \omega^k) \leq F(x^k; \rho_k, \mu^k, \omega^k) + \tau\beta^{\ell_k}\Delta_k,
    \end{equation*}
    where $\Delta_k := \inner{\grad F(x^k; \rho_k, \mu^k, \omega^k)}{p^k}$.
  Set $x^{k+1} := x^k + \beta^{\ell_k} p^k$.
  \State Compute $\mu^{k+1}, \omega^{k+1}, \phi_{k+1}, \psi_{k+1}$ and $\gamma_{k+1}$ by Algorithm \ref{alg:vomf}.
  \State If $\norm{\grad F(x^{k+1}; \rho^k, \mu^k, \omega^k)} \leq \gamma_k$, set $\rho_{k+1} := \rho_k/2$. Otherwise, set $\rho_{k+1} := \rho_k$.
  \State Set $k := k+1$ and go to Step 2.
  \end{algorithmic}
\end{algorithm}

\begin{theorem}
  \red{Let $\{x^k\}$ be a sequence generated by Algorithm~\ref{alg:sqp} and suppose that the following assumptions hold:}
  \begin{enumerate}[(a)]
  \item There exists a compact set which contains $\{x^k\}$.
  \item There exist positive constants $\nu_1$ and $\nu_2$ which satisfy
    \begin{equation*}
      \nu_1 \leq \lambda_{\mathrm{min}} \left(H_k + \frac{1}{\rho_k} \grad h(x^k) \grad h(x^k)^{\T}\right) \quad \mathrm{and} \quad \lambda_{\mathrm{max}}(H_k) \leq \nu_2,
    \end{equation*}
    for all $k$, where $\lambda_{\max}$ and $\lambda_{\min}$ mean the largest and the smallest eigenvalues of the corresponding matrices.
  \item Let $x^*$ be an accumulation point of $\{x^k\}$. \red{Then, $Q$ defined in~\eqref{eq:Q} satisfies the generalized Lojasiewicz inequality at $x^*$.}
  \item Let $x^*$ be an accumulation point of $\{x^k\}$ and $\{\epsilon_k\} \subset \R_{+}$ such that $\epsilon_k \to 0$. Then, for sufficiently large $k$, we obtain $I_{0, \epsilon_k}(x^k) = I_0(x^*)$ and $I_{B, \epsilon_k}(x^k) = I_B(x^*)$.
\end{enumerate}
  \red{Then, any accumulation point of $\{x^k\}$ either satisfies the AKKT2 \red{(and CAKKT2)} conditions, or it is an infeasible point of~\eqref{eq:socp}, but a stationary point of the feasibility measure~$Q$.}
\end{theorem}
  
\begin{proof}
  \red{The proof is similar to the one in~\cite[Theorem~6]{YA22}.} Suppose that the algorithm generates infinitely many V- or O-iterates and without loss of generality let $\{x^k\} \to x^*$. Then, both $r_V$ and $r_O$ converge to $0$. \red{Therefore, $x^*$ is a CAKKT point satisfying also the second-order condition of AKKT2. Since CAKKT implies AKKT~\cite[Theorem~4.2]{AN19}, $x^*$ satisfies AKKT2 too.} Even if the algorithm generates a finite number of V- or O-iterates, it is shown that a situation that the number of M-iterate is finite whereas that of F-iterate is infinite never occurs~\cite[Theorem 5]{YA22} and $\{x^k\}$ converges to a stationary point of the problem which is related to constraints of the original problem~\cite[Theorem 6]{YA22}.
\end{proof}

In the above theorem, observe that assumptions (a)--(c) are the same ones used in~\cite{YA22}. Assumption (d) indicates that $\{x^k\}$ is updated stably near limit points.

\section{Conclusion}
\label{sec:conclusion}
In this paper, we propose a second-order sequential optimality condition for SOCP called AKKT2, which extends the AKKT2 of NLP. We show that a local minimum of SOCP satisfies the condition under a weak assumption. We also consider its relation with the existing optimality condition WSONC and propose algorithms that generate AKKT2 points. \red{Moreover, similar results are obtained for CAKKT2.} As an already ongoing future work, we cite the proposal of AKKT2 conditions for nonlinear semidefinite programming and nonlinear symmetric 
conic optimization. It is also important to check the necessity of Assumption~\ref{assum:bounded} in the main theorem. \red{Another important work is to define strict CQs for our AKKT2, i.e., the weakest CQ that one can assume that would recover KKT.}


\section*{Acknowledgements} 
This work was supported by the Grant-in-Aid for Scientific Research (C) 
(19K11840) from Japan Society for the Promotion of Science. We are also grateful
to Yuya Yamakawa and Gabriel Haeser for the valuable discussions.



\printbibliography

\end{document}